\documentclass[11pt, a4paper]{amsart}
\pdfoutput=1
\usepackage[latin1]{inputenc} 
\usepackage{enumitem} 
\usepackage{amsmath}   
\usepackage{amsthm}   
\usepackage{amsfonts} 
\usepackage{amssymb}
\usepackage{mathtools}
\usepackage{stmaryrd}
\usepackage{relsize}   
\usepackage{tikz}    
\usetikzlibrary{patterns.meta}
\usepackage{subcaption}
\usepackage[runin, style]{abstract}
\usepackage{setspace}
\usepackage[pdfborder={0 0 0}]{hyperref} 
\def\www#1{\url{#1}}

\newcommand{\N}{\ensuremath{\mathbb{N}}}   
\newcommand{\R}{\ensuremath{\mathbb{R}}}

\newcommand{\Half}{\ensuremath{\mathcal{H}}}   

\DeclareMathOperator{\im}{im} 
\DeclareMathOperator{\Map}{Map} 
\DeclareMathOperator{\Cayl}{Cay}
\DeclareMathOperator{\supp}{supp}
\DeclareMathOperator{\Aut}{Aut}

\DeclareMathOperator{\Eorl}{(E^{or})^\ell}
\DeclareMathOperator{\Eor}{E^{or}}
\DeclareMathOperator{\ellalt}{\ell_{alt}^\infty(\mathcal{P},\R)}
\DeclareMathOperator{\cat0}{CAT(0)}

\newtheorem{theorem}{Theorem}[section]
\newtheorem*{theorem*}{Theorem}
\newtheorem{main}[theorem]{Main Theorem}
\newtheorem*{main*}{Main Theorem}
\newtheorem{lemma}[theorem]{Lemma}

\newtheorem*{corollary-non}{Corollary}
\theoremstyle{definition}
\newtheorem{definition}[theorem]{Definition}
\newtheorem{defthm}[theorem]{Definition/Theorem}
\newtheorem{notation}[theorem]{Notation}
\newtheorem{remark}[theorem]{Remark}
\newtheorem{setup}[theorem]{Setup}
{\begin{proof}[Beweis]}
	{\end{proof}}
\newtheorem{example}[theorem]{Example}

\newcommand{\qm}{quasi-median property}
\newcommand{\weight}{$\ell$-weight} 
\newcommand{\step}{$\ell$-fragment}
\newcommand{\stable}{$\Phi$-stable} 

\allowdisplaybreaks

\begin{document}
	
\title[A vanishing criterion for products in bounded cohomology]%
{\makebox[0pt]{A vanishing criterion for cup products and}\\
	 Massey products in bounded cohomology}

\author[]{Franziska Hofmann}
\address{Fakult\"{a}t f\"{u}r Mathematik, Universit\"{a}t Regensburg, Regensburg, Germany}
\email{franziska2.hofmann@ur.de}

\thanks{}

\keywords{bounded cohomology, cup product, quasimorphisms, median graphs}
\subjclass[2020]{20E08, 20F65, 18G90, 05E18, 57M60}

\date{\today.\ 
	This work was supported by the CRC~1085 \emph{Higher Invariants}
	(Universit\"at Regensburg, funded by the~DFG)}

\phantom{.}
\vspace{-\baselineskip}

\maketitle

\vspace{-1.2\baselineskip}

\begin{abstract}
	We give a uniform vanishing criterion for products in bounded cohomology. This allows us to reprove and extend previous vanishing results for cup products and Massey triple products in the bounded cohomology of free groups and in the equivariant bounded cohomology of group actions on $\cat0$ cube complexes.
\end{abstract}

\section{Introduction}

Bounded cohomology of groups has many interesting applications, but it is very hard to compute in general. One key example is the computation of bounded cohomology of non-abelian free groups $F$ with trivial real coefficients.
We know that bounded cohomology of $F$ is infinite dimensional in degree 2 via quasimorphisms~\cite{Brooks} and 3 via hyperbolic geometry~\cite{Soma}, but for now it is unknown whether it is trivial in higher degrees or not~\cite[Question~16.3]{BIMW}.
Classes in degree 2 are well studied as each of them is represented by the coboundary of a quasimorphism \cite[Section~2.6]{frigerio}.
Hence, one might try to use products with classes in degree 2 to understand classes in higher degrees. 
Although it is not clear whether cup products with arbitrary classes in degree~$2$ vanish, several weaker vanishing results are known for such products~\cite{BM,FF,Heuer,AB}.

We wish to highlight the following vanishing result of Amontova and Bucher~\cite{AB}:

\begin{theorem*}[{\cite[Theorem~A]{AB}}]
	Let $k > 0$ and $\alpha \in H_b^k(F;\R)$ be arbitrary. Let $\phi\colon F\to \R$ be a quasimorphism. Then the cup product
	\begin{align*}
		\cup \colon H_b^2(F;\R)\times H_b^k(F;\R) &\to H_b^{k+2}(F;\R),\\
		([\delta^1 \widehat{\phi}],\alpha) &\mapsto [0]
	\end{align*}
	vanishes whenever
	\begin{enumerate}
		\item $\phi$ is a $\Delta$-decomposable quasimorphism, or
		\item $\phi$ is a Brooks quasimorphism.
	\end{enumerate}
\end{theorem*}

The proof of this vanishing result uses aligned cochains. They provide a subcomplex of the bar resolution/inhomogeneous complex~\cite[Section~3]{AB}. The result has been generalised in several directions. First, Marasco~\cite{Marasco} considered Massey triple products in the bounded cohomology of non-abelian free groups. One can construct Massey products of three cohomology classes whenever the cup products of the first with the second and the second with the third class vanish. As the vanishing of cup products is known for classes represented by the coboundary of a $\Delta$-decomposable quasimorphism, he considers Massey triple products where the second class is given by the coboundary of such a quasimorphism and obtained the following result.

\begin{theorem*}[{\cite[Theorem~1]{Marasco}}]
	Let $\phi$ be a $\Delta$-decomposable quasimorphism, let $k_1,\, k_2>0$ and $\alpha_1 \in H_b^{k_1}(F;\R)$ and $\alpha_2\in H_b^{k_2}(F;\R)$. Then the Massey triple product
	$\langle \alpha_1, [\delta^1 \widehat{\phi}], \alpha_2\rangle$ is trivial.
\end{theorem*}

The proof of Marasco also uses the aligned cochain complex and explicit formulas given by Amontova and Bucher~\cite{AB}.

On the other hand, Br\"uck, Fournier-Facio and L\"oh~\cite{BFFL} generalised the notion of Brooks quasimorphisms and considered a special kind of counting quasimorphisms of group actions on $\cat0$ cube complexes, called median quasimorphisms. They obtained the following vanishing result:

\begin{theorem*}[{\cite[Theorem~3.23]{BFFL}}]
	Let $\Gamma$ be a group acting on a finite-dimensional $\cat0$ cube complex $X$ with finite staircase length. Let $s$ be an $\Half$-segment in $X$, and let $f_s$ be the corresponding median quasimorphism of $\Gamma\curvearrowright X$. Then, for every class $\zeta\in H_{\Gamma,b}^n(X;\R)$ that is non-transverse to
	the orbit $\Gamma s$, the cup product $[\delta^1f_s] \cup\zeta\in  H_{\Gamma,b}^{2+n}(X; \R)$ is trivial. 
\end{theorem*}

The definitions for finite staircase length, the quasimorphism $f_s$ and for non-transversality are recalled in Section~\ref{subsec:Median}.
The proof of this theorem has the same structure as the proof of Amontova and Bucher. However, it does not use aligned cochains, as it is unclear how aligned cochains can be generalised to cochains in the equivariant bounded cohomology of actions on $\cat0$ cube complexes.

\subsection*{A unified proof}
As the vanishing results mentioned above rely on the same blueprint, it should be possible to give a unified proof that covers and extends all of these results. For this we start by constructing \emph{weight quasimorphisms}~$f_\mathcal{W}$ of a group action on a graph $X=(V,E)$, where the value $f_\mathcal{W}(x,y)$ is given by a weighted sum of tuples of consecutive edges of a suitable path from $x$ to $y$. For a precise definition see Section~\ref{subsec:quasimorph}. We also give sufficient conditions that ensure that cup products and Massey products with classes given by weight quasimorphisms vanish. Those conditions and in particular the notion of $\Phi$-stable classes are explained in Section~\ref{sec:Vanishing}.

\begin{main*}[Theorems~\ref{main:trivialcupproduct} and~\ref{main:trivialMassey}]
	Let $\Gamma$ be a group acting on a graph $X=(V,E)$. Let $f_\mathcal{W}$ be a weight quasimorphism and let $\alpha_1\in H_{\Gamma,b}^{n}(X;\R)$ and $\alpha_2\in H_{\Gamma,b}^{m}(X;\R)$ be \stable\ classes. Then the cup products $\alpha_1\cup[\delta^1f_\mathcal{W}]$ and $[\delta^1f_{\mathcal{W}}]\cup \alpha_2$ as well as the Massey triple product $\langle \alpha_1, [\delta^1f_\mathcal{W}], \alpha_2 \rangle$ are trivial.
\end{main*}

Using the orbit map as a connection between equivariant bounded cohomology and bounded cohomology, this vanishing result can also be used to prove vanishing of products in bounded cohomology. In particular, our Main Theorem can be applied to 
\begin{itemize}
	\item (big) Brooks quasimorphisms, as we reprove the fact that their cup products vanish and also obtain the vanishing of their Massey products (Theorem~\ref{thm:Brooks}),
	\item $\Delta$-decomposable quasimorphisms in order to reprove the vanishing of their cup products and Massey products (Theorem~\ref{thm:Delta}),
	\item median quasimorphisms. \\ Here, we reprove the vanishing result for cup products with median quasimorphisms that was proven by Br\"uck, Fournier-Facio and L\"oh~\cite[Theorem~2.23]{BFFL} and extend the vanishing to Massey products.
\end{itemize}

\subsection*{Organization of the paper}
We start by collecting basics about (equivariant) bounded cohomology, products and quasimorphisms in Section~\ref{sec:Basics}. We then consider group actions on graphs with special properties and construct the weight quasimorphisms of this group action in Section~\ref{sec:construction}. Vanishing of products with the induced classes in degree $2$ are then studied in Section~\ref{sec:Vanishing}. In Section~\ref{sec:Application}, we apply the vanishing result to the setting of non-abelian free groups and actions on $\cat0$ cube complexes.

\subsection*{Acknowledgements}
I am very grateful to my advisor Clara L\"oh for suggesting this topic as a part of my master thesis and for many helpful discussions. I want to thank Talia Fern\'os for discussions on halfspaces in $\cat0$ cube complexes. I also wish to thank Jonathan Bowden and Francesco Fournier-Facio for several useful comments.

\section{Bounded Cohomology and quasimorphisms}
\label{sec:Basics}

In this section we collect various definitions, notation and statements about bounded cohomology and quasimorphisms. 
We start by giving a short introduction to equivariant bounded cohomology and products therein. Then we recall the definition of bounded cohomology of groups and highlight a connection to equivariant bounded cohomology.
For more details and proofs, we refer to~\cite{frigerio, Monod, Li}.

\subsection{Equivariant bounded cohomology}
\label{subsec:equivbddcohom}

As the name suggests, equivariant bounded cohomology is a bounded version of equivariant cohomology. 
We fix a (discrete) group $\Gamma$ and a set $S$ together with a group action $\Gamma\curvearrowright S$. Note that this induces an action of $\Gamma$ on $\ell^\infty (S^{n};\R)$ for $n\in\N$ via
\begin{align*}
	\gamma\cdot f\colon (s_1,\ldots, s_n) \mapsto f(\gamma^{-1}s_1,\ldots, \gamma^{-1}s_n)
\end{align*}
for $\gamma\in\Gamma$ and $ f\in \ell^\infty (S^{n},\R)$. 
We denote the $\Gamma$-invariant elements of this group action by $\ell^\infty (S^{n},\R)^\Gamma$.

\begin{definition}[equivariant bounded cohomology]
	We consider the chain complex $\big(C_{\Gamma,b}^*(S;\R), \delta^*\big)=\big(\ell^\infty (S^{*+1},\R)^\Gamma, \delta^*\big)$ where $\delta^*$ denotes the simplicial coboundary operator. Then the \emph{$\Gamma$-equivariant bounded cohomology of $S$ with coefficients in $\R$} is given by
	\begin{align*}
		H_{\Gamma,b}^*(S;\R) = H^*\big(C^*_{\Gamma,b}(S;\R)\big).
	\end{align*}
	For an action $\Gamma\curvearrowright X$ on a $\cat0$ cube complex or graph with vertex set~$V$ we also write
	\begin{align*}
		C_{\Gamma,b}^*(X;\R)\coloneqq C_{\Gamma,b}^*(V;\R),\text{ and}
		H_{\Gamma,b}^*(X;\R)\coloneqq H_{\Gamma,b}^*(V;\R).
	\end{align*}
\end{definition}

In this paper we focus on the following two products in equivariant bounded cohomology. They are defined in the usual fashion.

\begin{definition}[cup product]
	The \emph{cup product} on cochains of dimension~$p,q\in \N$ is given by 
	\begin{align*}
		\cup \colon C_{\Gamma,b}^p(S;\R) \otimes_\R C_{\Gamma,b}^q(S;\R) &\to C_{\Gamma,b}^{p+q}(S;\R)\\
		f\otimes g &\mapsto \big((s_0,\ldots, s_{p+q}) \mapsto f(s_0,\ldots, s_p)\cdot g(s_p,\ldots, s_{p+q})\big) .
	\end{align*}
	This map induces a well-defined map on the level of cohomology given by
	\begin{align*}
		\cup \colon H_{\Gamma,b}^p(S;\R) \otimes_\R H_{\Gamma,b}^q(S;\R) &\to H_{\Gamma,b}^{p+q}(S;\R)\\
		f\otimes g &\mapsto [f\cup g].
	\end{align*} 
\end{definition}

\begin{remark}
	We will often use the following equation: For $f\in H_{\Gamma,b}^p(S;\R)$ and $g\in H_{\Gamma,b}^q(S;\R)$ we have
	\begin{align*}
		\delta^{p+q}(f\cup g) = (\delta^p f)\cup g + (-1)^p f\cup (\delta^qg).
	\end{align*}
\end{remark}

\begin{definition}[Massey triple product]
	Let $p,q,k\in \N$, $\alpha_1\in H_{\Gamma,b}^p(S;\R)$, $\alpha_2\in H_{\Gamma,b}^q(S;\R)$, and $\alpha\in H_{\Gamma,b}^k(S;\R)$.
	If the cup products $\alpha_1\cup \alpha$ and $\alpha\cup \alpha_2$ are trivial, then the \emph{Massey triple product} $\langle \alpha_1,\alpha, \alpha_2\rangle$ is defined as the subset of $H_{\Gamma,b}^{p+q+k-1}(S;\R)$ consisting of the elements  $(-1)^k \left[(-1)^p\omega_1 \cup \beta_2 - \beta_1\cup \omega_2 \right]$ where
	\begin{itemize}
		\item the cocycles $\omega_1$ and $\omega_2$ are representatives for $\alpha_1$ and $\alpha_2$, respectively; and
		\item the cocycles $\beta_1$ and $\beta_2$ are primitives for $\omega_1\cup \omega$ and $\omega\cup\omega_2$, respectively, where $\omega$ is a cocycle representing $\alpha$.
	\end{itemize}
	We say the Massey triple product $\langle \alpha_1,\alpha, \alpha_2\rangle$ is \emph{trivial}, if it contains the trivial class.
\end{definition}

\begin{remark}
	Massey triple products have already been used in various settings. For example, there are interesting connections between Massey triple products in singular cohomology and the LS-category~\cite{Rud}. Furthermore, Massey triple products can be used to show that the Borromean rings are linked, as their complement has a non-trivial Massey product~\cite[pp. 85-88]{Viktor}.
\end{remark}

\subsection{Quasimorphisms of group actions}

\begin{definition}[quasimorphism of group action]
	\label{def:qmofaction}
	Let $\Gamma\curvearrowright S$ be a group action on a set $S$. A quasimorphism of $\Gamma \curvearrowright S$ is a map $f\colon S^2\to \R$ that is $\Gamma$-invariant and has finite \emph{defect}  
	\begin{align*}
		D(f)\coloneqq \|\delta^1 f\|_\infty = \sup_{x,y,z\in S} | f(y,z) - f(x,z)+ f(x,y)| <\infty.
	\end{align*}
\end{definition}

\begin{remark}
	Note that $f$ itself might not be a bounded map. But since $\delta^2\delta^1f = 0$ and $\delta^1f$ is bounded, we see that $\delta^1f$ defines a cocycle in equivariant bounded cohomology of $S$.
\end{remark}

\subsection{Bounded cohomology}

As we want to apply the vanishing criterion to bounded cohomology of groups, we recall its definition and give a short overview of its connection to equivariant bounded cohomology.

We fix a (discrete) group $\Gamma$. Using the construction in Section~\ref{subsec:equivbddcohom}, we see that the left translation action $\Gamma\curvearrowright\Gamma$ induces an action $\Gamma\curvearrowright\ell^\infty(\Gamma^n,\R)$ for $n\in\N$.

\begin{definition}[(bounded) cohomology of $\Gamma$]
	We consider the chain complex $\big(C_{b}^*(\Gamma;\R), \delta^*\big)=\big(\ell^\infty (\Gamma^{*+1},\R)^\Gamma, \delta^*\big)$ where $\delta^*$ denotes the simplicial coboundary operator. Then the \emph{bounded cohomology of $\Gamma$ with coefficients in $\R$} is given by
	\begin{align*}
		H_{b}^*(\Gamma;\R) = H^*\big(C^*_{b}(\Gamma;\R)\big).
	\end{align*}
\end{definition}

\begin{remark}
	Taking a closer look at the definition of equivariant bounded cohomology and bounded cohomology, we have $C_b^n(\Gamma;\R) = C_{\Gamma,b}^n(\Gamma;\R)$ when we consider the left translation action $\Gamma\curvearrowright\Gamma$. Hence, bounded cohomology is just a specialization of equivariant bounded cohomology. In particular, this allows us to transfer the definitions of cup products and Massey products to bounded cohomology.
\end{remark}

\subsection{Quasimorphisms}
Again, let $\Gamma$ be a group.

\begin{definition}[quasimorphism]
	A map $\varphi\colon \Gamma\to \R$ is called a \emph{quasimorphism of $\Gamma$} if it has finite \emph{defect}
	\begin{align*}
		D(\varphi)\coloneqq \sup_{g,h\in\Gamma} |\varphi(g)+\varphi(h)-\varphi(gh)| <\infty.
	\end{align*}
\end{definition}

\begin{example}
	Trivial examples for quasimorphisms of $\Gamma$ are group homomorphisms and bounded maps.
	Non-trivial examples for quasimorphisms of non-abelian free groups are Brooks quasimorphisms and $\Delta$-decomposable quasimorphisms introduced by Brooks~\cite{Brooks} and Heuer~\cite{Heuer}, respectively. We recall their definitions in Section~\ref{subsec:Brooks} and~\ref{subsec:Delta}, respectively.
\end{example}

\begin{remark}
	Let $\varphi$ be a quasimorphism of $\Gamma$. Then $\varphi$ defines a class in the second bounded cohomology of $\Gamma$ in the following way.
	We define
	\begin{align*}
		\widehat{\varphi}\colon \Gamma^2&\to \R\\
		(g,h)&\mapsto \varphi(g^{-1}h).
	\end{align*}
	Then $\widehat{\varphi}$ and thus also $\delta^1\widehat{\varphi}$ are $\Gamma$-invariant and the latter is bounded by $D(\varphi)$. Hence, we have $[\delta^1\widehat{\varphi}]\in H_b^2(\Gamma;\R)$.
\end{remark}

\subsection{Connection between bounded cohomology and equivariant \texorpdfstring{\newline}{ } bounded cohomology}
\label{subsec:Connection}

Let $\Gamma$ be a group acting on a set $S$. For an element $s\in S$, we consider the orbit map
\begin{align*}
	o_s\colon \Gamma &\to S\\
	g&\mapsto g\cdot s.
\end{align*}
This map induces a $\Gamma$-equivariant cochain map $o_s^*\colon C_{\Gamma,b}^* (S;\R) \to C_b^*(\Gamma;\R),$ given by
\begin{align*}
	o_s^n\colon C_{\Gamma,b}^n (S;\R) &\to C_b^n(\Gamma;\R)\\
	f&\mapsto ((g_0,\ldots, g_n) \mapsto f(g_0s,\ldots, g_ns)).
\end{align*}
in dimension $n\in\N$.
The induced map in bounded cohomology will also be denoted by 
\begin{align*}
	o_s^*\colon H^*_{\Gamma,b} (S;\R) \to H_b^*(\Gamma;\R).
\end{align*}

\begin{remark}[products]
	\label{rem:productsconnection}
	One can easily verify that the orbit map $o_s^*$ is compatible with products in bounded cohomology in the following way:
	For $s\in S$, $p,q,k\in \N$ and classes $\alpha_1\in H_{\Gamma,b}^p(S;\R)$, $\alpha_2\in H_{\Gamma,b}^q(S;\R)$, and $\alpha\in H_{\Gamma,b}^k(S;\R)$ we have
	\begin{align*}
		o_s^{p+q}(\alpha_1\cup \alpha_2)& = \big(o_s^p(\alpha_1)\big)\cup \big(o_s^q(\alpha_2)\big),\quad \text{and} \\
		o_s^{p+q+k-1}\big(\langle \alpha_1,\alpha, \alpha_2\rangle\big) &\subset \big\langle o_s^p(\alpha_1),o_s^k(\alpha), o_s^q(\alpha_2)\big\rangle.
	\end{align*}
\end{remark}

\begin{remark}[quasimorphisms]
	\label{rem:conn_quasimorphisms}	
	Let $f$ be a quasimorphism of $\Gamma\curvearrowright S$. For $s\in S$, the map
	\begin{align*}
		f_s\colon \Gamma&\to \R\\
		g&\mapsto f(s,gs)
	\end{align*}
	is a quasimorphism of $\Gamma$ and 
	\begin{align*}
		o_s^2([\delta^1 f]) = [\delta^1 \widehat{f_s}].
	\end{align*}
\end{remark}

When the action has amenable stabilizers, then the orbit map is an isomorphism:

\begin{theorem}{\cite[Theorem~4.23]{frigerio}}
	\label{thm:orbitmap_is_iso}
	Suppose the stabilizer of every element of $S$ is amenable. Then, 
	\begin{align*}
		o_s^*\colon H^*_{\Gamma,b} (S;\R) \to H_b^*(\Gamma;\R)
	\end{align*}
	is an isomorphism.
\end{theorem}

\section{Construction of weight quasimorphisms}
\label{sec:construction}

This section is structured as follows:
Firstly, we recall the definition of graphs and paths to fix notation. We then introduce the quasi-median property for graphs and, associated to it, a property of the group action on the graph. Finally we construct quasimorphisms of this group action.

\subsection*{Graphs and paths}

In this paper, we always work with undirected simplicial graphs. Let $X=(V,E)$ be such an undirected graph and let $n \in \N$. A path in $X$ of length $n$ is a sequence $v_0,\ldots,v_n$ of \emph{distinct} vertices of $X$ with the property that $\lbrace v_j,v_{j+1}\rbrace\in  E$ holds for all $j \in\lbrace 0,\ldots,n-1\rbrace$. Let $p= v_0,\ldots, v_n$ and $q=w_0,\ldots, w_m$ be two paths in $X$ such that $p$ and~$q$ share no vertex except from $v_n=w_0$. Then the \emph{concatenation $p\ast q$ of $p$ and~$q$} and the \emph{inverse path $\overline{p}$ of $p$} are given by
\begin{align*}
	p\ast q &= v_0,\ldots, v_n, w_1,\ldots w_m,\\
	\overline{p} &= v_n,\ldots, v_0.
\end{align*} 
For two vertices $v,w\in V$ we denote by $[v,w]$ the set of geodesics from $v$ to~$w$, i.e., the set of shortest paths from $v$ to $w$.
Furthermore, we denote by~$\Eor$ the set of oriented edges of $X$,
\begin{align*}
	\Eor \coloneqq \left\lbrace (\alpha,\omega)\in V^2\mid \lbrace\alpha,\omega\rbrace \in E\right\rbrace.
\end{align*}
For an oriented edge $e=(\alpha,\omega)\in \Eor$, we denote by $\overline{e}$ the corresponding edge with the opposite orientation, i.e., $\overline{e}\coloneqq (\omega,\alpha)$. Moreover, for $\ell\in\N_{>0}$ and a tuple $a=(e_1,\ldots,e_\ell)\in \Eorl$ of oriented edges, we define $\overline{a}\coloneqq (\overline{e_\ell},\ldots,\overline{e_1})$.

\begin{definition}
	Consider a finite path $p=x_0,\ldots,x_n$ in $X$ and its oriented edges $e_i=(x_i,x_{i+1})$ for $i\in \lbrace 1,\ldots,n-1\rbrace$.
	For $\ell\in\N$, we call a tuple  $(e_{i_1},\ldots, e_{i_\ell})$ with $0\leq i_1<i_2<\ldots<i_\ell\leq n-1$ an \emph{\step}\ of $p$. The set of \step s of $p$ is denoted by $p^{(\ell)}$.
	
	For an \step\ $a=(e_{i_1},\ldots, e_{i_\ell})\in p^{(\ell)}$ we define its \emph{head} $\alpha(a)$ and its \emph{tail} $\omega(a)$ via
	\begin{align*}
		\alpha(a)\coloneqq x_{i_1} \quad\text{ and }\quad \omega(a) \coloneqq x_{i_\ell+1}.
	\end{align*} 
	
	We say a vertex $m$ that appears on the path $p$ \emph{is contained in $a\in p^{(l)}$} if $\alpha(a)\neq m\neq \omega(a)$ and if the vertices $\alpha(a)$, $m$, $\omega(a)$ appear in this order on the path $p$.
\end{definition}

In order to obtain vanishing results, we introduce the quasi-median property for families of paths.

\begin{definition}[quasi-median property]
	\label{def:quasimedprop}
	Let $X=(V,E)$ be a connected graph.
	For $x,y\in V$, let $P(x,y)\neq\emptyset$ be a set of finite paths from $x$ to $y$.
	The family $(P(x,y))_{x,y\in V}$ is said to fulfil the \emph{\qm\ for $R\in\N$}, if the following holds:
	
	For all $x,y,z\in V$ there exist
	\begin{itemize}
		\item a triple $(m_x,m_y,m_z)\in V^3$, and
		\item paths $p_{xy}\in P(x,y)$, $p_{yz}\in P(y,z)$, and $p_{xz}\in P(x,z)$
	\end{itemize}
	that allow a decomposition
	\begin{align*}
		p_{xy} &= s_x*r_1*\overline{s_y},\\
		p_{yz} &= s_y*r_2*\overline{s_z},\\
		p_{xz} &= s_x*r_3*\overline{s_z}
	\end{align*}
	where $s_x\in P(x,m_x)$, $s_y\in P(y,m_y)$, $s_z\in P(z,m_z)$ and  $r_1\in P(m_x,m_y)$, $r_2\in P(m_y,m_z)$, $r_3\in P(m_x,m_z)$ such that $r_1$, $r_2$, and $r_3$ have length at most $R$, see Figure~\ref{pic:qmprop}.
	\\ \\
	We say that $(P(x,y))_{x,y\in V}$ fulfils the \emph{\qm}, if there exists some $R\in \N$ such that the family fulfils the \qm\ for~$R$.
\end{definition}

\begin{figure}[h] 
	\begin{center}
		\begin{tikzpicture}[>=stealth, scale = 0.9]
			\newcommand{\arrowIn}{\tikz \draw[-stealth] (-1pt,0) -- (1pt,0);}
			
			\coordinate (a) at (-1,0);
			\coordinate (b) at (1,0);
			\coordinate (c) at (0,1.6);
			
			\coordinate (A) at (-2.5,-0.75);
			\coordinate (B) at (2.5,-0.75);
			\coordinate (C) at (0,3);
			
			\filldraw (a) circle (1pt);
			\node[below] at (a) {\scriptsize$m_x$};
			\filldraw (b) circle (1pt);
			\node[below] at (b) {\scriptsize$m_z$};
			\filldraw (c) circle (1pt);
			\node[left] at (c) {\scriptsize$m_y$};
			
			\filldraw (A) circle (1pt);
			\node[below left] at (A) {$x$};
			\filldraw (B) circle (1pt);
			\node[below right] at (B) {$z$};
			\filldraw (C) circle (1pt);
			\node[above=2.5pt] at (C) {$y$};
			
			\draw[-] (b) to node[above] {\scriptsize$r_3$} node{\arrowIn} (a);
			\draw[-] (a) to node[sloped,below,allow upside down] {\scriptsize$r_1$} node[sloped, pos=0.5, allow upside down]{\arrowIn} (c);
			\draw[-] (c) to node[sloped,below] {\scriptsize$r_2$} node[sloped, pos=0.5, allow upside down]{\arrowIn} (b);
			
			\draw[-]  (A) to node[sloped,above] {\scriptsize$s_x$} node[sloped, pos=0.5, allow upside down]{\arrowIn} (a);
			\draw[-]  (B) to node[sloped,above] {\scriptsize$s_z$} node[sloped, pos=0.5, allow upside down]{\arrowIn} (b);
			\draw[-]  (C) to node[right=1pt] {\scriptsize$s_y$} node[sloped, pos=0.5, allow upside down]{\arrowIn} (c);
			
			\draw [-,dashed] (-2.4,-1) .. controls (-1,-0.5) and (1,-0.5) .. node[sloped,below] {\scriptsize$p_{xz}$} node[sloped, pos=0.5, allow upside down]{\arrowIn} (2.4,-1); 
			\draw [-,dashed] (-2.6,-0.45) .. controls (-1.5,0) and (-0.5,1.6) .. node[sloped,above] {\scriptsize$p_{xy}$} node[sloped, pos=0.5, allow upside down]{\arrowIn} (-0.3,3); 
			\draw [-,dashed] (0.3,3) .. controls (0.5,1.6) and (1.5,0) ..   node[sloped,above] {\scriptsize$p_{yz}$} node[sloped, pos=0.5, allow upside down]{\arrowIn} (2.6,-0.45); 
			
		\end{tikzpicture}
	\end{center}
	\caption{Quasi-median property} \label{pic:qmprop}
\end{figure}
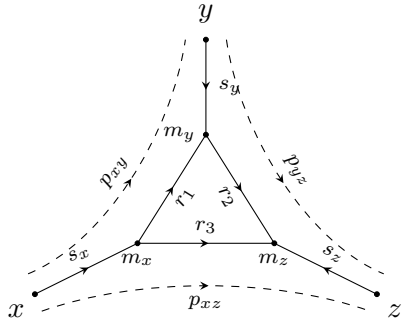

\begin{remark}
	This definition generalises the notion of median graphs. A graph $X=(V,E)$ with graph metric $D$ is called median graph if for all $x,y,z\in V$, there exits a unique vertex $m\in V$ such that 
	\begin{align*}
		D(x,y) &= D(x,m) + D(m,y)\\
		D(x,z) &= D(x,m) + D(m,z)\\
		D(y,z) &= D(y,m) + D(m,z).
	\end{align*}
	Equivalently, one can say that the family $\big([x,y]\big)_{x,y\in V}$ of geodesics fulfils the quasi-median property for $R=0$. 
	There are also other approaches to generalise the notion of median graphs, for example almost median graphs or coarse median spaces that admit a median point for every triple of vertices up to a uniformly bounded error~\cite{Bowditch, FG}. Our approach differs from those as it is not a property of the graph/space itself but a property of a family of arbitrary paths in the given graph. We will also use families that fulfil the quasi-median property but do not necessarily consist of geodesics when we apply the general setting to $\Delta$-decomposable quasimorphisms~\ref{subsec:Delta}.
\end{remark}

Now let $\Gamma$ be a discrete group acting on the graph $X$ via graph automorphisms. 

\begin{definition}[coherent pair]
	\label{def:coherentpair}
	For $x,y\in V$, let $P(x,y)\neq\emptyset$ be a set of finite paths from $x$ to $y$ such that the family $P=(P(x,y))_{x,y\in V}$ fulfils the quasi-median property. We say that $P$ is \emph{coherent} if
	\begin{enumerate}
		\item $P$ is compatible with the action of $\Gamma$, i.e., for all $g\in \Gamma$ we have $g P(x,y) = P(g x, g y)$;
		\item $P$ is compatible with inversion, i.e., we have $P(x,y) = \overline{P(y,x)}$;
		\item $P$ is closed under taking sub-paths, i.e., for $p\in P(x,y)$, every sub-path $p'$ of $p$ is contained in $P(x',y')$ for $x'$ the starting point of $p'$ and $y'$ its endpoint; and
		\item \label{item:coherentlength} for all $x,y\in V$ any two paths $p,q\in P(x,y)$ have the same length. 
	\end{enumerate}
	Let $\ell \in \N_{>0}$. For vertices $x,y\in V$ and paths $p,q\in P(x,y)$ we fix a bijection $\varphi_{p,q}\colon p^{(\ell)}\to q^{(\ell)}$ and define
	\begin{align*}
		\Phi\coloneqq (\varphi_{p,q})_{(p,q)\in \bigcup\limits_{{(x,y)\in V^2}} P(x,y)^2}
	\end{align*}
	Then we call the pair $(P,\Phi)$ a \emph{coherent pair of size $\ell$}.
\end{definition}

\begin{remark}
	For $x,y\in V$ and $p,q\in P(x,y)$ there is an obvious bijection between $p^{(l)}$ and $q^{(l)}$ induced by a map that sends the $i$-th edge of $p$ to the $i$-th edge of~$q$.  
	In the applications, we sometimes use these obvious bijections. But in order to (re)prove statements involving median quasimorphisms it is necessary to consider more complex bijections, see Remark~\ref{rem:why}.
\end{remark}

\subsection*{Weight quasimorphisms} We will now construct quasimorphisms for a group action on a graph that admits a coherent pair.
\label{subsec:quasimorph}

\begin{setup}
	\label{setup}
	We consider a group $\Gamma$ acting on a graph $X=(V,E)$ via graph automorphisms. Furthermore, we assume there exists $\ell\in \N_{>0}$ and a coherent pair 
	\begin{align*}
		(P,\Phi)= \big((P(x,y))_{x,y\in V}, (\varphi_{p,q})_{(p,q)\in \bigcup\limits_{{(x,y)\in V^2}} P(x,y)^2}\big)
	\end{align*}
	of size $\ell$. 
\end{setup}

\begin{definition}[\weight]
	Assuming we have a coherent pair $(P,\Phi)$ of size $\ell$ as in Setup~\ref{setup} a map $\mathcal{W}\colon \Eorl\to \R$ is called an \emph{\weight}\ if 
	\begin{enumerate}
		\item it is $\Gamma$-invariant with respect to the diagonal action;
		\item it is alternating, i.e., the equation $-\mathcal{W}(a) =\mathcal{W}(\overline{a})$ holds for all $a\in \Eorl$;
		\item it is bounded with respect to the supremum norm;
		\item it is path-independent with respect to $\Phi$, i.e. for all $x,y\in V$ and $p,q\in P(x,y)$ we have
		\begin{align*}
			\mathcal{W}_{\mid p^{(\ell)}} = \mathcal{W}_{\mid q^{(\ell)}}\circ \varphi_{p,q}\text{ ; and}
		\end{align*}
		\item \label{item:finitenesscond} it fulfils the following finiteness condition: There exists~$c\in \N_{>1}$ such that for all $x,y\in V$, for all $p\in P(x,y)$, and for all vertices $m$ on $p$ we have
		\begin{align*}
			\Big| \big\lbrace a\in p^{(\ell)} \mid m \text{ is contained in } a \big\rbrace \cap \supp(\mathcal{W})\Big|\leq c
		\end{align*}
		where $\supp(\mathcal{W}) \coloneqq \big\lbrace a\in\Eorl \mid \mathcal{W}(a) \neq 0\big\rbrace$.
	\end{enumerate}
\end{definition}

\begin{defthm}[weight quasimorphism]
	\label{main:quasimorphi} 
	In the situation of Setup~\ref{setup} let $\mathcal{W}\colon  \Eorl\to \R$ be an \weight. 
	Then the map 
	\begin{align*}
		V\times V &\to \R\\
		(x,y) &\mapsto \sum_{a\in p^{(l)}}\mathcal{W}(a),
	\end{align*}
	with $p\in P(x,y)$, is a well-defined, antisymmetric quasimorphism of $\Gamma\curvearrowright X$. We call this map the \emph{weight quasimorphism for $\mathcal{W}$}.
\end{defthm}

\begin{remark}
	Note that these quasimorphisms are not always trivial, since weight quasimorphisms are a generalization of (big) Brooks, $\Delta$-decomposable and median quasimorphisms (Section~\ref{sec:Application}). 
	It is also necessary that $\supp(\mathcal{W})$ can contain fragments whose edges are not directly consecutive in order to show that median quasimorphisms are weight quasimorphisms, see Remark~\ref{rem:why}.
\end{remark}

For the proof of Theorem~\ref{main:quasimorphi}, we need the definition of $\Phi$-stable maps and a lemma.

\begin{definition}[$\Phi$-stable map]
	\label{def:Phistablemap}
	Assume we are in the situation of Setup~\ref{setup}.  A map $\tau\colon \Eorl\to \R$ is then called \emph{$\Phi$-stable} if for all $x,y\in V$, $p,q\in P(x,y)$ and $a\in p^{(l)}\cap \supp(\mathcal{W})$ we have
	\begin{align*}
		\tau(a) = \tau(\varphi_{p,q}(a)).
	\end{align*}
\end{definition}

The notion of $\Phi$-stability is similar to path-independence. However, $\Phi$-stability is a weaker condition, as it requires path-independence only on $\supp(\mathcal{W})$.

\begin{lemma}
	\label{lemma:Hilfe}
	In the situation of Setup~\ref{setup}, let $R\in\N$ such that the coherent family $P$ fulfils the \qm\ for $R$. Furthermore, let $\mathcal{W}\colon  \Eorl\to \R$ be an \weight\ and $\tau\colon \Eorl\to \R$ be a map that is
	\begin{itemize}
		\item symmetric, i.e., the equation $\tau(a) = \tau(\overline{a})$ holds for all $a\in \Eorl$; and
		\item $\Phi$-stable. 
	\end{itemize}
	Then for all triples of vertices $x,y,z\in V$ and all paths $p_{xy}\in P(x,y)$, $p_{yz}\in P(y,z)$, and $p_{xz}\in P(x,z)$ the absolute value 
	\begin{align*}
		\bigg|\sum_{{a\in p_{xy}^{(\ell)}}} \mathcal{W}(a)\tau(a) + \sum_{{a\in p_{yz}^{(\ell)}}} \mathcal{W}(a)\tau(a) - \sum_{{a\in p_{xz}^{(\ell)}}} \mathcal{W}(a)\tau(a) \bigg| 
	\end{align*}
	is bounded by $3 (R+1) c \|\mathcal{W}\|_\infty \|\tau\|_\infty$, where the constant $c$ witnesses the finiteness property of the \weight\ $\mathcal{W}$ (Property~\ref{item:finitenesscond}).
\end{lemma}

\begin{proof}
	The $\Phi$-stability of $\tau$ implies that for $x,y\in V$ and $p,q\in P(x,y)$ we have $(\mathcal{W}\cdot \tau)_{\mid p^{(\ell)}}=(\mathcal{W}\cdot \tau)_{\mid q^{(\ell)}}\circ \varphi_{p,q}$. Since $\varphi_{p,q}$ is a bijection, we have
	\begin{align*}
		\sum_{\mathclap{a\in p^{(\ell)}}} \mathcal{W}(a)\tau(a) &= \sum_{\mathclap{a\in q^{(\ell)}}} \mathcal{W}(a)\tau(a).
	\end{align*}
	This means that, for every choice of vertices $x,y,z\in V$, it suffices to prove the claim for paths $p_{xy}\in P(x,y)$, $p_{yz}\in P(y,z)$, and $p_{xz}\in P(x,z)$ that have a decomposition into subpaths
	\begin{align*}
		p_{xy} &= s_x*r_1*\overline{s_y},\\
		p_{yz} &= s_y*r_2*\overline{s_z},\\
		p_{xz} &= s_x*r_3*\overline{s_z},
	\end{align*}
	where the length of $r_1$, $r_2$, and $r_3$ is bounded by $R$, see Figure~\ref{pic:qmprop}.
	The existence of such paths is guaranteed by the fact that $P$ fulfils the quasi-median property for $R$.
	
	At first we observe that $p_{xy}^{(\ell)}$ can be written as the following disjoint (!) union:
	\begin{align*}
		p_{xy}^{(\ell)} = s_x^{(\ell)} \sqcup \overline{s_y}^{(\ell)} \sqcup  A_{xy}
	\end{align*} 
	with
	\begin{align*}
		A_{xy} \coloneqq p_{xy}^{(\ell)}\setminus  \big(s_x^{(\ell)} \cup\overline{ s_y}^{(\ell)}\big) = \bigcup_{m\in r_1\cap V} \big\lbrace a\in p_{xy}^{(\ell)} \mid m\text{ is contained in } a\big\rbrace.
	\end{align*}
	Since there are at most $R+1$ vertices in $r_1\cap V$, the finiteness condition of the \weight\ $\mathcal{W}$ gives rise to the inequality
	\begin{align}
		\label{lem:finitenessproof}
		\big|A_{xy} \cap \supp{\mathcal{W}}\big| \leq (R+1)\cdot c.
	\end{align}
	In the same way we define the sets $A_{yz}$ and $A_{xz}$ and since the preceding observation was independent of the choice of vertices, we obtain the same upper bound for the cardinalities of $A_{yz} \cap \supp{\mathcal{W}}$ and $A_{xz} \cap \supp{\mathcal{W}}$.
	
	Furthermore, as $\tau$ is symmetric and $\mathcal{W}$ is alternating, we obtain
	\begin{align*}
		\sum_{\mathclap{a\in p_{xy}^{(\ell)}}} \mathcal{W}(a)\tau(a)
		&= \sum_{\mathclap{a\in s_x^{(\ell)}}} \mathcal{W}(a)\tau(a) + \sum_{\mathclap{a\in \overline{s_y}^{(\ell)}}} \mathcal{W}(a)\tau(a) + \sum_{\mathclap{a\in A_{xy}}} \mathcal{W}(a)\tau(a) \\
		&= \sum_{\mathclap{a\in s_x^{(\ell)}}} \mathcal{W}(a)\tau(a) - \sum_{\mathclap{a\in s_y^{(\ell)}}} \mathcal{W}(a)\tau(a) + \sum_{\mathclap{a\in A_{xy}}} \mathcal{W}(a)\tau(a).
	\end{align*}
	Using Inequality~\ref{lem:finitenessproof}, the last sum is bounded by~$(R+1) c \|\mathcal{W}\|_\infty\|\tau\|_\infty$, which is independent of $x,y$.
	Applying the same arguments to $\sum_{{a\in p_{yz}^{(\ell)}}} \mathcal{W}(a)\tau(a)$ and $\sum_{{a\in p_{xz}^{(\ell)}}} \mathcal{W}(a)\tau(a)$ we obtain the following equation up to a finite error bounded by $3(R+1) c \|\mathcal{W}\|_\infty\|\tau\|_\infty$: 
	\begin{align*}
		\sum_{\mathclap{a\in p_{xy}^{(l)}}} \mathcal{W}(a)\tau(a) + \sum_{\mathclap{a\in p_{yz}^{(l)}}} \mathcal{W}(a)&\tau(a) - \sum_{\mathclap{a\in p_{xz}^{(l)}}} \mathcal{W}(a)\tau(a) \\
		=&\quad\sum_{\mathclap{a\in s_x^{(l)}}} \mathcal{W}(a)\tau(a) - \sum_{\mathclap{a\in s_y^{(l)}}} \mathcal{W}(a)\tau(a)\\
		&+\sum_{\mathclap{a\in s_y^{(l)}}} \mathcal{W}(a)\tau(a) - \sum_{\mathclap{a\in s_z^{(l)}}} \mathcal{W}(a)\tau(a)\\
		&-\sum_{\mathclap{a\in s_x^{(l)}}} \mathcal{W}(a)\tau(a) + \sum_{\mathclap{a\in s_z^{(l)}}} \mathcal{W}(a)\tau(a)\\
		=& \quad0 \qedhere
	\end{align*}
\end{proof}

\begin{proof}[Proof of Theorem~\ref{main:quasimorphi}]
	Let $\mathcal{W}\colon  \Eorl\to \R$ be an \weight. 
	The map
	\begin{align*}
		f_{\mathcal{W}}\colon V\times V &\to \R\\
		(x,y) &\mapsto \sum_{\mathclap{a\in p^{(\ell)}}}\mathcal{W}(a),
	\end{align*}
	is independent of the choice of path $p\in P(x,y)$ by the path-independence of~$\mathcal{W}$. Furthermore, we have $P(x,y) = \overline{P(y,x)}$ for all $x,y\in V$. Hence,~$f_\mathcal{W}$ is antisymmetric as this is the case for~$\mathcal{W}$.
	In order to prove that~$f_\mathcal{W}$ is a quasimorphism, we need to show~$\Gamma$-invariance and boundedness of the defect. The first property follows from the~$\Gamma$-invariance of~$\mathcal{W}$ and the coherence of~$P$ with the group action. 
	\\
	For boundedness, we consider the constant map $\tau = 1_\R$ on $\Eorl$ with value~$1$. The map $\tau$ is clearly symmetric and $\Phi$-stable. Hence, we can apply Lemma~\ref{lemma:Hilfe} and obtain $3(R+1) c \|\mathcal{W}\|_\infty$ as a finite upper bound for the defect of $f_\mathcal{W}$.
\end{proof}

\section{Vanishing of products}
\label{sec:Vanishing}

In order to obtain vanishing results for products with the classes induced by weight quasimorphisms, we consider the following property for cochains and classes in equivariant bounded cohomology. 

\begin{setup}
	\label{setup2}
	We consider a group $\Gamma$ acting on a graph $X=(V,E)$ via graph automorphisms. We assume that there exists a coherent pair
	\begin{align*}
		(P,\Phi) = \big((P(x,y))_{x,y\in V},  (\varphi_{p,q})_{(p,q)\in \bigcup\limits_{{(x,y)\in V^2}} P(x,y)^2}\big)
	\end{align*} of size $\ell\in\N_{>0}$. 
	We furthermore fix an \weight\ $\mathcal{W}$ and its corresponding weight quasimorphism
	\begin{align*}
		f_\mathcal{W}\colon V\times V &\to \R\\
		(x,y) &\mapsto \sum_{a\in p^{(\ell)}}\mathcal{W}(a)
	\end{align*}
	with $p\in P(x,y)$.
\end{setup}

\begin{definition}[\stable\ cochain/class]
	In the situation of Setup~\ref{setup2}, a cochain $\zeta\in C_{\Gamma,b}^n(X;\R)$ is said to be \emph{\stable}\ if for all $x_1,\ldots,x_n\in V$ the maps
	\begin{align*}
		\zeta(\alpha(\cdot),x_1,\ldots,x_n) \colon \Eorl&\to \R\\
		\zeta(\omega(\cdot),x_1,\ldots,x_n) \colon \Eorl&\to \R
	\end{align*}
	are $\Phi$-stable in the sense of Definition~\ref{def:Phistablemap}.
	
	A class $\alpha\in H_{\Gamma,b}^n(X;\R)$ is said to be \emph{\stable}\ if it admits a \stable\ representative.
\end{definition}

Now we are ready to formulate vanishing results for cup products and Massey triple products with classes given by weight quasimorphisms.

\begin{main}[Triviality of the cup product]
	\label{main:trivialcupproduct}
	In the situation of Setup~\ref{setup2}, let $\alpha\in H_{\Gamma,b}^n(X;\R)$ be \stable. Then the cup products $[\delta^1f_{\mathcal{W}}]\cup \alpha$ and $\alpha\cup[\delta^1f_\mathcal{W}]$ are trivial in $H_{\Gamma,b}^{n+2}(X;\R)$.
\end{main}

\begin{main}[Triviality of the Massey triple product]
	\label{main:trivialMassey}
	In the situation of Setup~\ref{setup2}, let $\alpha_1\in H_{\Gamma,b}^{n}(X;\R)$ and $\alpha_2\in H_{\Gamma,b}^{m}(X;\R)$ be \stable\ classes. Then the Massey triple product $\langle \alpha_1, [\delta^1f_\mathcal{W}], \alpha_2 \rangle\subset H_{\Gamma,b}^{n+m+1}(X;\R)$ is trivial.
\end{main}

\begin{proof}[Proof of Theorem~\ref{main:trivialcupproduct}]
	In order to prove the vanishing result for cup products, let $\alpha\in H_{\Gamma,b}^n(X;\R)$ be a \stable\ class and $\zeta \in  C_{\Gamma,b}^n(X;\R)$ be a \stable\ representative of $\alpha$.
	In order to show that the cup product~$[\delta^1f_\mathcal{W}]\cup \alpha$ is trivial, we need to find a $\Gamma$-invariant map $\eta \in \Map(X,\R)$ such that $\beta\coloneqq f_\mathcal{W}\cup \zeta + \delta^n\eta$ is a bounded $(n+1)$-cochain. We can then obtain the triviality of the cup product via
	\begin{align*}
		[\delta^1f_\mathcal{W}]\cup \alpha = [\delta^1 f_\mathcal{W} \cup \zeta] = [\delta^{n+1}\beta] = 0.
	\end{align*}

	In order to construct a suitable candidate for $\eta$, we use a similar strategy to that of Br\"uck, Fournier-Facio and L\"oh \cite{BFFL} and define 
	\begin{align*}
		\tilde{\zeta} \colon \Eorl\times V^n&\to \R\\
		(a,x_1,\ldots,x_n)&\mapsto \frac{1}{2} \big(\zeta\big(\alpha(a),x_1,\ldots, x_n\big) + \zeta\big(\omega(a),x_1,\ldots, x_n\big) \big).
	\end{align*}
	Since $\zeta$ is a \stable\ cocycle, the map $\tilde{\zeta}(\cdot,x_1,\ldots,x_n)$ for fixed $x_1,\ldots, x_n\in V$ is $\Phi$-stable.
	For this reason, the map 
	\begin{align*}
		\eta\colon V^{n+1}&\to \R\\
		(x_0,\ldots, x_n) &\mapsto \sum_{\mathclap{a\in p_{01}^{(\ell)}}}\mathcal{W}(a) \tilde{\zeta}(a,x_1,\ldots, x_n)
	\end{align*}
	with $p_{01}\in P(x_0,x_1)$ is well-defined as it does not depend on the choice of path $p_{01}\in P(x_0,x_1)$. By the $\Gamma$-invariance of $\mathcal{W}$ and $\zeta$, it follows that the map $\eta$ is $\Gamma$-invariant.
	In order to prove boundedness of $\beta\coloneqq f_\mathcal{W}\cup \zeta + \delta^n\eta$ we fix vertices $x_0,\ldots, x_{n+1}\in V$ and paths $p_{01}\in P(x_0,x_1)$, $p_{02}\in P(x_0,x_2)$ and $p_{12}\in P(x_1,x_2)$. We now wish to apply Lemma~\ref{lemma:Hilfe} to the symmetric and $\Phi$-stable map $\tau\coloneqq \tilde{\zeta}(\cdot, x_2,\ldots,x_{n+1})$.
	
	Using the cocycle condition for $\zeta$ we obtain the following equation for an arbitrary $a\in \Eorl$:
	\begin{align*}
		\tau(a) &= \tilde{\zeta}(a,x_2,\ldots,x_{n+1}) \\
		&=\zeta(x_1,\ldots, x_{n+1}) + \sum_{i=2}^{n+1} (-1)^i \tilde{\zeta}(a,x_1,x_2,\ldots,\widehat{x_i},\ldots, x_{n+1})
	\end{align*}
	This allows us to compute
	\begin{align*}
		\beta&(x_1,\ldots,x_{n+1}) = f_\mathcal{W}(x_0,x_1)\cdot \zeta(x_1,\ldots,x_{n+1}) + \delta^n\eta(x_0,\ldots, x_{n+1})\\
		= &\sum_{\mathclap{a\in p_{01}^{(\ell)}}} \mathcal{W}(a) \zeta (x_1,\ldots,x_{n+1}) + \sum_{i=2}^{n+1}(-1)^i\eta(x_0,x_1,x_2,\ldots, \widehat{x_i},\ldots, x_{n+1})\\
		&+ \eta(x_1,\ldots, x_{n+1}) - \eta(x_0,x_2,\ldots,x_{n+1})\\
		= &\sum_{\mathclap{a\in p_{01}^{(\ell)}}} \mathcal{W}(a) \Big( \zeta(x_1,\ldots,x_{n+1}) + \sum_{i=2}^{n+1} (-1)^i \tilde{\zeta}(a,x_1,x_2\ldots,\widehat{x_i},\ldots, x_{n+1})\Big) \\
		&+ \eta(x_1,\ldots, x_{n+1}) - \eta(x_0,x_2,\ldots,x_{n+1})\\
		=& \sum_{\mathclap{a\in p_{01}^{(\ell)}}} \mathcal{W}(a)\tau(a) + \sum_{\mathclap{a\in p_{12}^{(\ell)}}} \mathcal{W}(a)\tau(a) - \sum_{\mathclap{a\in p_{02}^{(\ell)}}} \mathcal{W}(a)\tau(a).
	\end{align*}
	Applying Lemma~\ref{lemma:Hilfe} gives an upper bound for $\|\beta\|_\infty$ which is finite, since $\|\tau\|_\infty\leq \|\zeta\|_\infty<\infty$.
	\newline\newline
	Just like in other cohomologies, the cup product in bounded cohomology is also graded commutative. Hence we can conclude that also $\alpha\cup[\delta^1 f_\mathcal{W}]$ is trivial for a \stable\ class $\alpha\in H_{\Gamma,b}^n(X;\R)$. Alternatively, one can prove the triviality of $\alpha\cup[\delta^1 f_\mathcal{W}]$ in a symmetric way: For a \stable\ representative $\zeta$ of $\alpha$ one considers the cochain
	\begin{align*}
		\nu\colon V^{n+1}&\to \R\\
		(x_0,\ldots,x_n)&\mapsto \sum_{\mathclap{a\in p_{n-1,n}^{(\ell)}}} \mathcal{W}(a) \tilde{\zeta}(a,x_0,\ldots,x_{n-1})
	\end{align*}
	where $p_{n-1,n}\in P(x_{n-1},x_n)$.
	One can show that the map $\beta'\coloneqq (-1)^n \left(\zeta\cup f_\mathcal{W} - \delta^n\nu\right)$ is a cochain in $C_{\Gamma,b}^{n+1}(X;\R)$ with $\delta^{n+1} \beta' = \zeta\cup \delta^1f_\mathcal{W}$. This proves that the cup product $\alpha\cup[\delta^1 f_\mathcal{W}] = [\zeta \cup \delta^1f_\mathcal{W}] = [\delta^{n+1}\beta']=0$ is trivial.
\end{proof}

\begin{proof}[Proof of Theorem~\ref{main:trivialMassey}]
	This proof is based on the same blueprint as the proof of Marasco~\cite{Marasco}.
	Let $\alpha_1\in H_{\Gamma,b}^{n}(X;\R)$ and $\alpha_2\in H_{\Gamma,b}^{m}(X;\R)$ be \stable\ classes with \stable\ representatives $\zeta_1$ and $\zeta_2$, respectively.
	
	As in the proof for the triviality of the cup product we consider the $\Gamma$-invariant maps
	\begin{align*}
		\nu\colon V^{n+1}&\to \R\\
		(x_0,\ldots,x_n)&\mapsto \sum_{\mathclap{a\in p_{n-1,n}^{(\ell)}}} \mathcal{W}(a) \tilde{\zeta_1}(a,x_0,\ldots,x_{n-1})
	\end{align*}
	with $p_{n-1,n}\in P(x_{n-1},x_n)$, and
	\begin{align*}
		\eta\colon V^{m+1}&\to \R\\
		(x_0,\ldots, x_m) &\to \sum_{\mathclap{a\in p_{0,1}^{(\ell)}}}\mathcal{W}(a) \tilde{\zeta_2}(a,x_1,\ldots, x_m)
	\end{align*}
	with $p_{0,1}\in P(x_0,x_1)$.
	In the proof of Theorem~\ref{main:trivialcupproduct}, we showed that the cochains $\beta_1\coloneqq (-1)^n\left(\zeta_1\cup f _\mathcal{W}- \delta^n\nu \right)$ and $\beta_2\coloneqq f_\mathcal{W}\cup \zeta_2 + \delta^m\eta$ are bounded primitives for $[\zeta_1\cup \delta^1f]$ and $[\delta^1 f\cup \zeta_2]$, respectively. Hence,
	\begin{align*}
		\left[(-1)^n\zeta_1\cup \beta_2 - \beta_1\cup\zeta_2\right]\in \langle \alpha_1, \left[\delta^1f_\mathcal{W}\right], \alpha_2\rangle
	\end{align*}
	and we focus on proving that this class is zero.
	
	For this we take a closer look at the given representative and obtain
	\begin{align*}
		(-1)^n\zeta_1\cup \beta_2 - \beta_1\cup\zeta_2 &= (-1)^n \zeta_1\cup \delta^m\eta +  (-1)^n \delta^n\nu \cup \zeta_2 \\
		&= \delta^{n+m} \big(\zeta_1\cup \eta + (-1)^n \nu \cup \zeta_2\big).
	\end{align*}
	The goal is now to find a $\Gamma$-invariant map $\kappa\colon V^{n+m}\to \R$ such that 
	\begin{align*}
		\beta\coloneqq \zeta_1\cup \eta + (-1)^n \nu\cup \zeta_2 - \delta^{n+m-1}\kappa
	\end{align*}
	is a bounded map so that its coboundary $(-1)^n\zeta_1\cup \beta_2 - \beta_1\cup\zeta_2$ represents the trivial class in $H_{\Gamma,b}^{n+m+1}(X;\R)$. 
	
	We consider 
	\begin{align*}
		\kappa \colon V^{n+m} &\to \R\\
		(x_1,\ldots, x_{n+m}) &\mapsto \sum_{\mathclap{a\in p_{n,n+1}^{(\ell)}}} \mathcal{W}(a) \tilde{\zeta_1}(a,x_1,\ldots, x_n) \tilde{\zeta_2}(a,x_{n+1},\ldots ,x_{n+m})
	\end{align*}
	with $p_{n,n+1}\in P(x_{n}, x_{n+1})$.
	Again, this map does not depend on the choice of path in $P(x_{n}, x_{n+1})$ as $\mathcal{W}$ is path-independent and $\zeta_1$ and $\zeta_2$ are \stable. Since $\mathcal{W}, \zeta_1$ and $\zeta_2$ are $\Gamma$-invariant, the same holds for $\kappa$. So we know that the map 
	\begin{align*}
		\beta\coloneqq \zeta_1\cup \eta + (-1)^n \nu\cup \zeta_2 - \delta^{n+m-1}\kappa
	\end{align*}
	is $\Gamma$-invariant and it remains to show that it is bounded.
	
	To do this, we fix $x_0,\ldots,x_{n+m}\in V$. 
	For a better readability, we denote by $p_{i,j}$ an element in $P(x_i,x_j)$ for $i,j\in \lbrace 0,\ldots, n+m\rbrace$ and introduce the maps
	\begin{align*}
		\tau_1, \tau_2, \tau\colon \Eorl &\to \R\\
		\tau_1(a) &= \tilde{\zeta_1}(a,x_0,\ldots, x_{n-1})\\
		\tau_2(a) &=  \tilde{\zeta_2}(a,x_{n+1},\ldots, x_{n+m})\\
		\tau(a) &= \tau_1(a)\cdot \tau_2(a),
	\end{align*}
	which are all $\Phi$-stable.
	Since $\zeta_1$ and $\zeta_2$ fulfil the cocycle condition, we obtain for all $a\in\Eorl$
	\begin{align*}
		\zeta_1(x_0,\ldots, x_n)- (-1)^n\tau_1(a)  &=  \sum_{i=0}^{n-1} (-1)^i \tilde{\zeta_1}(a,x_0,\ldots,\widehat{x_i},\ldots, x_{n-1}, x_n) 
	\end{align*}
	and
	\begin{align*}
		\zeta_2(x_n,\ldots, x_{n+m}) -\tau_2(a) &=  \sum_{i=1}^{m} (-1)^i  \tilde{\zeta_2}(a,x_{n}, x_{n+1},\ldots, \widehat{x_{n+i}}, \ldots ,x_{n+m})\\
		&= (-1)^n \sum_{i=n+1}^{n+m} (-1)^i  \tilde{\zeta_2}(a,x_{n}, x_{n+1},\ldots, \widehat{x_{i}}, \ldots ,x_{n+m}).
	\end{align*}
	
	Furthermore, we can simplify
	\begin{align*}
		\nu(x_0,\ldots, x_n) &= \sum_{\mathclap{a\in p_{n-1,n}^{(\ell)}}} \mathcal{W}(a) \tau_1(a);\\
		\eta(x_n,\ldots,x_{n+m})&= \sum_{\mathclap{a\in p_{n,n+1}^{(\ell)}}} \mathcal{W}(a) \tau_2(a).
	\end{align*}
	Taking a look at the coboundary of $\kappa$, we notice that omitting an entry $i\in\lbrace 0,\ldots, n+m\rbrace$ can be divided in three different cases. 
	\begin{enumerate}
		\item For $i\in\lbrace 0,\ldots,n-1\rbrace$ we have
		\begin{align*}
			\kappa(x_0,\ldots, \widehat{x_i},\ldots, x_{n+m}) &=  \sum_{\mathclap{a\in p_{n,n+1}^{(\ell)}}} \mathcal{W}(a) \tilde{\zeta_1}(a,x_0,\ldots,\widehat{x_i},\ldots, x_{n-1}, x_n) \tau_2(a).
		\end{align*}
		\item For $i\in\lbrace n+1,\ldots, n+m\rbrace$ we have
		\begin{align*}
			\kappa (x_0,\ldots, \widehat{x_i},\ldots, x_{n+m}) &= \sum_{\mathclap{a\in p_{n-1,n}^{(\ell)}}} \mathcal{W}(a) \tau_1(a)\tilde{\zeta_2}(a,x_{n}, x_{n+1},\ldots, \widehat{x_i}, \ldots ,x_{n+m}).
		\end{align*}
		\item In the remaining case $i=n$ we have
		\begin{align*}
			\kappa (x_0,\ldots, \widehat{x_n},\ldots, x_{n+m}) &= \sum_{\mathclap{a\in p_{n-1,n+1}^{(\ell)}}}\mathcal{W}(a) \tau_1(a) \tau_2(a).
		\end{align*}
	\end{enumerate}
	
	These three equations allow us to simplify the coboundary of $\kappa$:
	
	\begin{align*}
		\delta^{n+m-1}\kappa&(x_0,\ldots,x_{n+m}) \\
		=& \sum_{i=0}^{n-1} \sum_{{a\in p_{n,n+1}^{(\ell)}}} (-1)^i  \mathcal{W}(a) \tilde{\zeta_1}(a,x_0,\ldots,\widehat{x_i},\ldots, x_{n-1}, x_n) \tau_2(a)\\
		& +  \sum_{i=n+1}^{n+m}\sum_{{a\in p_{n-1,n}^{(\ell)}}} (-1)^i \mathcal{W}(a) \tau_1(a)\tilde{\zeta_2}(a,x_{n}, x_{n+1},\ldots, \widehat{x_i}, \ldots ,x_{n+m})\\			
		& + (-1)^n \sum_{{a\in p_{n-1,n+1}^{(\ell)}}}\mathcal{W}(a) \tau_1(a) \tau_2(a)\\
		=& \sum_{\mathclap{a\in p_{n,n+1}^{(\ell)}}} \mathcal{W}(a) \big(\zeta_1(x_0,\ldots, x_n)- (-1)^n\tau_1(a)  \big)\tau_2(a)\\
		& +\sum_{\mathclap{a\in p_{n-1,n}^{(\ell)}}}  \mathcal{W}(a) \tau_1(a)(-1)^n\big( \zeta_2(x_n,\ldots, x_{n+m}) -\tau_2(a) \big)\\
		& + (-1)^n \sum_{\mathclap{a\in p_{n-1,n+1}^{(\ell)}}}\mathcal{W}(a) \tau(a)\\
		=& 	\sum_{\mathclap{a\in p_{n,n+1}^{(\ell)}}} \mathcal{W}(a) \zeta_1(x_0,\ldots, x_n)  \tau_2(a) - (-1)^n\sum_{\mathclap{a\in p_{n,n+1}^{(\ell)}}}\mathcal{W}(a)\tau(a)\\
		& +(-1)^n\sum_{\mathclap{a\in p_{n-1,n}^{(\ell)}}}  \mathcal{W}(a) \tau_1(a) \zeta_2(x_n,\ldots, x_{n+m}) - (-1)^n\sum_{\mathclap{a\in p_{n-1,n}^{(\ell)}}}  \mathcal{W}(a) \tau(a)\\
		& + (-1)^n \sum_{\mathclap{a\in p_{n-1,n+1}^{(\ell)}}}\mathcal{W}(a) \tau(a)\\
		=& \zeta_1\cup \eta (x_0,\ldots,x_{n+m}) + (-1)^n\nu\cup \zeta_2 (x_0,\ldots,x_{n+m})\\
		& - (-1)^n\Big(\sum_{\mathclap{a\in p_{n,n+1}^{(\ell)}}}\mathcal{W}(a)\tau(a) +\sum_{\mathclap{a\in p_{n-1,n}^{(\ell)}}}  \mathcal{W}(a) \tau(a) - \sum_{\mathclap{a\in p_{n-1,n+1}^{(\ell)}}}\mathcal{W}(a) \tau(a)\Big)
	\end{align*}
	
	As $\tau$ is symmetric and $\Phi$-stable, we obtain that $\beta$ is bounded by using Lemma~\ref{lemma:Hilfe}, since
	\begin{align*}
		\beta(x_0,\ldots,x_{n+m}) = &\zeta_1\cup \eta (x_0,\ldots,x_{n+m})+ (-1)^n \nu\cup \zeta_2 (x_0,\ldots,x_{n+m})\\
		&- \delta^{n+m-1}\kappa (x_0,\ldots,x_{n+m})\\
		=& (-1)^n\Big( \sum_{\mathclap{a\in p_{n,n+1}^{(\ell)}}}\mathcal{W}(a)\tau(a) +\sum_{\mathclap{a\in p_{n-1,n}^{(\ell)}}}  \mathcal{W}(a) \tau(a) - \sum_{\mathclap{a\in p_{n-1,n+1}^{(\ell)}}}\mathcal{W}(a) \tau(a)\Big). \qedhere
	\end{align*}
\end{proof}

\section{Applications}
\label{sec:Application}

\subsection*{Preliminaries on Cayley graphs}

In order to obtain vanishing results for the bounded cohomology of a group $\Gamma$ via our results on cup products and Massey products in equivariant bounded cohomology, Theorems~\ref{main:trivialcupproduct} and~\ref{main:trivialMassey}, we want to find a graph on which $\Gamma$ acts so that we can use the connection of equivariant bounded cohomology and bounded cohomology described in Section~\ref{subsec:Connection}. 

A canonical choice is a Cayley graph for the given group. We briefly recall its definition and give a short overview on the natural group action. For more details, we refer to the literature, e.g., \cite{Loeh}.

\begin{definition}[Cayley graph]
	Let $\Gamma$ be a group with generating set $S\subset \Gamma$.
	The \emph{Cayley graph of $\Gamma$ with respect to $S$} is the graph $\Cayl(\Gamma,S)$ having $\Gamma$ as its set of vertices and the set 
	\begin{align*}
		\big\lbrace \lbrace g,gs\rbrace \mid g\in \Gamma \text{ and } s\in (S\cup S^{-1})\setminus\lbrace e\rbrace \big\rbrace
	\end{align*}
	as its set of edges.
\end{definition}

\begin{example}
	\label{ex:freegroup-tree}
	Let $F$ be a non-abelian free group freely generated by $S\subset F$. Then $\Cayl(F,S)$ is a tree~\cite[Theorem~3.3.1]{Loeh}.
\end{example}

\begin{lemma}
	Let $\Gamma$ be a group with generating set $S\subset \Gamma$. Then $\Gamma$ acts by graph automorphisms on the Cayley graph $\Cayl(\Gamma,S)$ via left translation,
	\begin{align*}
		\Gamma & \to \Aut\big(\Cayl(\Gamma,S)\big)\\
		g & \mapsto (h\mapsto g\cdot h).
	\end{align*}
\end{lemma}

\subsection*{Preliminaries on median graphs}
For applications we will consider median graphs. We now recall the definition and some examples.

\begin{definition}[median graph]
	Let $X=(V,E)$ be a graph with corresponding graph metric $D$. Then $X$ is called a \emph{median graph} if for all $x,y,z\in V$, there exits a unique vertex $m\in V$ such that 
	\begin{align*}
		D(x,y) &= D(x,m) + D(m,y)\\
		D(x,z) &= D(x,m) + D(m,z)\\
		D(y,z) &= D(y,m) + D(m,z).
	\end{align*}
	In other words, the vertex $m$ is the unique vertex lying simultaneously on a geodesic in $[x,y]$, $[x,z]$, and $[y,z]$.
\end{definition}

\begin{example}
	Examples for median graphs are trees, the square grid or, more generally, 1-skeletons of $\cat0$ cube complexes~\cite{Genevois,Hagen}
\end{example}

\subsection{Brooks quasimorphisms}
\label{subsec:Brooks}

Let $F$ be a non-abelian free group and let~$S$ be a free generating set for $F$. In the following, we interpret an element of $F$ as a (unique) reduced words over $S\cup S^{-1}$.
Let $\omega\in F$ be a non-empty reduced word of length $\ell$. 

Then we define
\begin{align*}
	\chi_\omega \colon F &\to \lbrace -1,0,+1\rbrace\\
	g&\mapsto\begin{cases} +1, & \textrm{if } g=w,\\ -1, & \textrm{if }  g=w^{-1},\\ \:\:\,\,0, & \textrm{otherwise}. \end{cases}
\end{align*}
The (big) Brooks quasimorphism on $\omega$ is then defined as
\begin{align*}
	\phi_\omega\colon F & \to \R\\
	g&\mapsto \sum_{i=1}^{m-\ell+1} \chi_\omega(s_i\cdots s_{i+\ell-1})
\end{align*}
where $g= s_1\cdots s_m$ as reduced word over $S\cup S^{-1}$.
In other words, the (big) Brooks quasimorphism on $\omega$ counts for $g\in F$ the number of occurrences of~$\omega$ in $g$ minus the number of occurrences of $\omega^{-1}$ in $g$.

(Big) Brooks quasimorphisms were introduced by Robert Brooks~\cite{Brooks}. They provide lots of examples for non-trivial quasimorphisms. More precisely, there exists an infinite family of Brooks quasimorphisms such that their corresponding classes in $H_b^2(F;\R)$ are linearly independent~\cite[Exampe~2.1]{Heuer}\cite{Brooks}. Furthermore, Grigorchuk~\cite[Theorem~5.7]{Grigorchuk} proved that the subspace spanned by the classes given by (big) Brooks quasimorphisms lies dense in $H_b^2(F;\R)$ for a suitable topology of pointwise convergence.

\begin{remark}
	There are also small Brooks quasimorphisms introduced by  Epstein and Fujiwara~\cite{EF} counting only non-overlapping occurrences of $\omega$. 
	We only focus on big Brooks quasimorphisms as they have easier combinatorics. So in the following, we will call them Brooks quasimorphisms for short.
\end{remark}

The following can be shown using weight quasimorphisms and the vanishing results for their products, Theorems~\ref{main:trivialcupproduct} and~\ref{main:trivialMassey}.

\begin{theorem}
	\label{thm:Brooks}
	Let $F$ be a non-abelian free group, freely generated by $S$ and let $\omega\in F$ be a non-empty reduced word of length $\ell$.
	The Brooks quasimorphism $\phi_\omega$ is a quasimorphism of $F$.
	
	Furthermore, for all $\alpha_1\in H_{b}^{n}(F;\R)$ and $\alpha_2\in H_{b}^{m}(F;\R)$, the cup products  $\alpha_1 \cup [\delta^1\phi_\omega]$ and $[\delta^1\phi_\omega]\cup \alpha_2$ as well as the Massey triple product $\langle \alpha_1, [\delta^1\phi_\omega], \alpha_2 \rangle$ are trivial.
\end{theorem}

\begin{proof}
	We consider the Cayley graph $X=\Cayl(F,S)$ together with the left translation action of~$F$ on $X$ via left translation. We know that $X$ is a tree and hence it is uniquely geodesic and a median graph.
	For two vertices $x,y$ in $\Cayl(F,S)$, we define $P(x,y)\coloneqq [x,y]$ and consider the family $P=\big(P(x,y)\big)_{x,y\in F}$. Then $P$ fulfils the quasi-median property for $R=1$ since $X$ is uniquely geodesic and a median graph.
	We define $\Phi$ to be the family containing the identity map on $[x,y]^{(\ell)}$ for $x,y\in V$. Then $(P,\Phi)$ is a coherent pair of size $\ell$ and by the definition of $\Phi$, we do not need to worry about path-independence for \weight s and furthermore, every map from $\Eorl$ to $\R$ is $\Phi$-stable. 
	
	We can assign to every oriented edge $e=(\alpha,\omega)$ the label 
	\begin{align*}
		\lambda(e) = \alpha^{-1} \omega\in S\cup S^{-1}. 
	\end{align*}
	We call a tuple $(e_1,\ldots, e_\ell)\in \Eorl$ \emph{connected}, if for all $i\in \lbrace 1,\ldots, \ell-1\rbrace $ the equation $\omega(e_i) = \alpha(e_{i+1})$ holds.
	
	We now define
	\begin{align*}
		\mathcal{W}\colon \Eorl & \to \R\\
		(e_1,\ldots,e_\ell) &\mapsto \begin{cases}
			\chi_\omega \big(\lambda(e_1)\cdots \lambda(e_\ell)\big) &\text{ if } (e_1,\ldots,e_\ell) \text{ is connected},\\
			0&\text{ otherwise},
		\end{cases}
	\end{align*}
	and claim that this map is an \weight. 
	
	We check that the following holds for $\mathcal{W}$:
	\begin{enumerate}
		\item $\Gamma$-invariance: This is true, since the label of an oriented edge is $\Gamma$-invariant;
		\item boundedness, as $\chi_\omega$ is bounded by $1$;
		\item the alternating property as for an oriented edge $e$ it is $\lambda(\overline{e}) = \lambda(e)^{-1}$ and $\chi_\omega$ is alternating;
		\item path-independence, as $\Phi$ only consists of identity maps; and
		\item the finiteness property: For this, we note that for $x,y\in V$, there is a 1-to-1 correspondence between connected tuples in $[x,y]^{(\ell)}$ and subpaths of length $\ell$ of $[x,y]$. A connected tuple in $[x,y]^{(\ell)}$ contains a vertex $m$ in its interior, if and only if the corresponding subpath of $[x,y]$ passes $m$ and if $m$ does not correspond with one of the endpoints. Since there are at most $\ell-1$ such subpaths we obtain
		\begin{align*}
			&\left| \left\lbrace a\in [x,y]^{(\ell)} \mid m \text{ is contained in } a \right\rbrace \cap \supp(\mathcal{W})\right|\leq \ell-1. 
		\end{align*}
	\end{enumerate}
	
	Let $f_\mathcal{W}$ denote the corresponding weight quasimorphism. We reprove using Remark~\ref{rem:conn_quasimorphisms} that the Brooks quasimorphism~$\phi_\omega$ is a quasimorphism of $F$ as it coincides with the quasimorphism $(f_\mathcal{W})_e$. This can be seen by comparing for an element $g\in F$ the labels of the edges of the geodesic from $e$ to $g$ with the letters of $g$ as a reduced word in $S\cup S^{-1}$.
	
	Since every class $\alpha\in H_{F,b}^k\big(X;\R\big)$ is \stable, we can apply Theorems~\ref{main:trivialcupproduct} and~\ref{main:trivialMassey} to deduce the vanishing of the products with $[\delta^1f_\mathcal{W}]$ in $H_{F,b}^*(X;\R)$.
	In order to obtain the vanishing of products in $H_b^*(F;\R)$, we firstly notice using Theorem~\ref{thm:orbitmap_is_iso} that the orbit map
	\begin{align*}
		o^2_* \colon H_{F,b}^*(X;\R) \to H_b^*(F;\R)
	\end{align*}
	is an isomorphism as the stabiliser of an arbitrary vertex is trivial.
	Since the orbit map preserves cup products and Massey triple products, Remark~\ref{rem:productsconnection}, we obtain the vanishing results for the products in $H_b^*(F;\R)$.
\end{proof}

\subsection{\texorpdfstring{${\Delta}$}{Delta}-decomposable quasimorphisms}
\label{subsec:Delta}

In this section, we discuss $\Delta$-decomposable quasimorphisms, which were first introduced by Heuer~\cite{Heuer, AB}. 
Again, let $F$ be a non-abelian free group with free generating set $S$ and elements of $F$ shall be interpreted as reduced words over $S\cup S^{-1}$.

\begin{notation}[sequences]
	Let $A\subset F$ be a symmetric subset, i.e., $a^{-1}\in A$ if $a\in A$. We denote by $A^\ast$ the set of finite sequences in $A$, including the empty sequence. For $s=(a_1,\ldots,a_n)\in A^\ast$ we define $n$ to be the length of $s$ and we denote by $s^{-1}$ the sequence $(a_n^{-1},\ldots,a_1^{-1})$. By symmetry of~$A$ we have $s^{-1}\in A^\ast$. For a sequence $t=(b_1,\ldots, b_m)\in A^\ast$ we define the \emph{common sequence of $s$ and $t$} to be the sequence $(a_1,\ldots,a_r)$, where $r$ is the largest integer with $r\leq \min\lbrace m,n\rbrace$ such that $a_j = b_j$ for all $j\leq r$. By convention this is empty if $a_1\neq b_1$.
	Moreover, we define the \emph{concatenation $s\cdot t$ of $s$ and $t$} by $s\cdot t\coloneqq (a_1,\ldots,a_n,b_1,\ldots,b_m)$. Note that this is the unreduced concatenation of sequences (in contrast to the concatenation of reduced words in the free group). 
\end{notation}

\begin{definition}[$\Delta$-decomposition]
	\label{def:Deltadecomp}
	Let $\mathcal{P}\subset F$ be a symmetric subset, called \emph{pieces of $F$}, not containing the neutral element. A \emph{$\Delta$-decomposition of $F$ into the pieces $\mathcal{P}$} is a map $\Delta\colon F\to \mathcal{P}^\ast$ assigning to every $g\in F$ a sequence $(g_1,\ldots,g_n)\in \mathcal{P}^\ast$ such that
	\begin{enumerate}
		\item \label{prop:Delta_genset} the reduced expression of $g$ in the letters $S\cup S^{-1}$ is given by the concatenation (without cancellation) of the words $g_1,\ldots, g_n$,
		\item the sequence $\Delta(g^{-1})$ is given by $\big(\Delta(g)\big)^{-1}$, and
		\item for all $i,j\in \lbrace 1,\ldots, n\rbrace$ with $i\leq j$ we have $\Delta(g_i\cdots g_j) = (g_i,\ldots, g_j)$.
	\end{enumerate}
	Furthermore, we require the existence of a constant $R\in \N$ with the following property:
	\begin{enumerate}
		\item For all $g,h\in F$ let 
		\begin{itemize}
			\item $c_1\in \mathcal{P}^\ast$ be such that $c_1^{-1}$ is the common sequence of $\Delta(g)$ and~$\Delta (gh)$,
			\item $c_2\in \mathcal{P}^\ast$ be such that $c_2^{-1}$ is the common sequence of $\Delta(g^{-1})$ and $\Delta (h)$,
			\item $c_3\in \mathcal{P}^\ast$ be such that $c_3^{-1}$ is the common sequence of $\Delta(h^{-1})$ and $\Delta \big((gh)^{-1}\big)$.
		\end{itemize}
		Let $r_1,r_2, r_3\in \mathcal{P}^\ast$ be the sequences uniquely determined by 
		\begin{align*}
			\Delta(g) &= c_1^{-1} \cdot r_1\cdot c_2,\\
			\Delta(h) &= c_2^{-1} \cdot r_2\cdot c_3,\\
			\Delta(gh) &= c_1^{-1} \cdot r_3^{-1}\cdot c_3.
		\end{align*}
		Then the length of $r_1,r_2,r_3$ is bounded by $R$.
	\end{enumerate}
	For the pair $(g,h)$ we call $c_1,c_2,c_3$ the \emph{$c$-part of the $\Delta$-triangle} of $(g,h)$ and $r_1,r_2,r_3$ the \emph{$r$-part of the $\Delta$-triangle} of $(g,h)$.
\end{definition}

We want to mention that the $\Delta$-decomposition $(g_1,\ldots,g_n)$ of a word $g\in F$ into pieces $\mathcal{P}$ is not required to be a decomposition into pieces $\mathcal{P}$ of shortest length. In other words, in the Cayley graph $\Cayl(F;\mathcal{P})$, the path from $e$ to $g$ that is given by the $\Delta$-decomposition might not be a geodesic.

\begin{figure}[h] 
	\begin{center}
		\begin{tikzpicture}[>=stealth]
			\newcommand{\arrowIn}{\tikz \draw[-stealth] (-1pt,0) -- (1pt,0);}
			
			\coordinate (a) at (-1,0);
			\coordinate (b) at (1,0);
			\coordinate (c) at (0,1.6);
			
			\coordinate (A) at (-2.5,-0.75);
			\coordinate (B) at (2.5,-0.75);
			\coordinate (C) at (0,3);
			
			\draw[-] (b) to node[above] {\scriptsize$r_3$} node[sloped, pos=0.5, allow upside down]{\arrowIn} (a);
			\draw[-] (a) to node[sloped,below,allow upside down] {\scriptsize$r_1$} node[sloped, pos=0.5, allow upside down]{\arrowIn} (c);
			\draw[-] (c) to node[sloped,below] {\scriptsize$r_2$} node[sloped, pos=0.5, allow upside down]{\arrowIn} (b);
			
			\draw[-]  (a) to node[sloped,above] {\scriptsize$c_1$} node[sloped, pos=0.5, allow upside down]{\arrowIn} (A);
			\draw[-]  (b) to node[sloped,above] {\scriptsize$c_3$} node[sloped, pos=0.5, allow upside down]{\arrowIn} (B);
			\draw[-]  (c) to node[right=1pt] {\scriptsize$c_2$} node[sloped, pos=0.5, allow upside down]{\arrowIn} (C);
			
			\draw [-,dashed] (-2.4,-1) .. controls (-1,-0.5) and (1,-0.5) .. node[sloped,below] {\scriptsize$\Delta(gh)$} node[sloped, pos=0.5, allow upside down]{\arrowIn} (2.4,-1); 
			\draw [-,dashed] (-2.6,-0.45) .. controls (-1.5,0) and (-0.5,1.6) .. node[sloped,above] {\scriptsize$\Delta(g)$} node[sloped, pos=0.5, allow upside down]{\arrowIn} (-0.3,3); 
			\draw [-,dashed] (0.3,3) .. controls (0.5,1.6) and (1.5,0) ..   node[sloped,above] {\scriptsize$\Delta(h)$} node[sloped, pos=0.5, allow upside down]{\arrowIn} (2.6,-0.45); 
			
		\end{tikzpicture}
	\end{center}
	\caption{$\Delta$-decomposition} \label{pic:decomposition}
\end{figure}
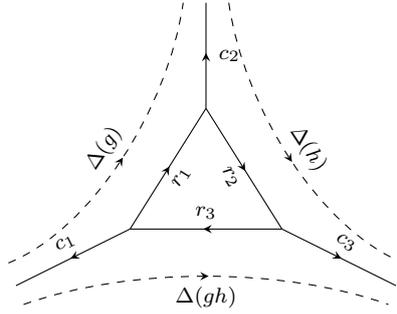

\begin{definition}[$\Delta$-decomposable quasimorphism]
	Let $\mathcal{P}\subset F$ be a symmetric subset not containing the neutral element. Let $\Delta\colon F\to \mathcal{P}^\ast$ be a $\Delta$-decomposition of $F$ into pieces $\mathcal{P}$ and let $\lambda\in \ellalt$ be an alternating bounded map on $\mathcal{P}$, i.e., $\|\lambda\|_\infty <\infty$ and $\lambda(p^{-1}) = -\lambda(p)$ for all $p\in \mathcal{P}$. For $g\in F$ we write $\Delta(g)= (g_1,\ldots,g_n)$ for its decomposition into pieces. Then the map
	\begin{align*}
		\phi_{\lambda,\Delta}\colon F&\to \R\\
		g&\mapsto \sum_{i=0}^{n}\lambda(g_i)
	\end{align*}
	is called a \emph{$\Delta$-decomposable quasimorphism}.
\end{definition}

Again, there exists an infinite family of $\Delta$-decomposable quasimorphisms whose corresponding family of classes in $H_b^2(F;\R)$ is linearly independent~\cite[Examples~2.2 and~3.9]{Heuer}. So $\Delta$-decomposable quasimorphisms again provide many examples for non-trivial quasimorphisms of $F$.

The following is a consequence of the vanishing of products with classes induced by weight quasimorphisms, Theorems~\ref{main:trivialcupproduct} and~\ref{main:trivialMassey}. In particular, Theorem~\ref{thm:Delta} combines the results of Amontova and Bucher~\cite[Theorem~A]{AB} and Marasco~\cite[Theorem~1]{Marasco}.

\begin{theorem}
	\label{thm:Delta}
	Let $\Delta\colon F\to \mathcal{P}^\ast$ be a $\Delta$-decomposition of $F$ into a symmetric set $\mathcal{P}$ and let $\lambda\in \ellalt$. Then $\phi_{\lambda,\Delta}$ is a quasimorphism of $F$.
	
	Furthermore, for all $\alpha_1\in H_{b}^{n}(F;\R)$ and $\alpha_2\in H_{b}^{m}(F;\R)$, the cup products  $\alpha_1 \cup [\delta^1\phi_{\lambda,\Delta}]$ and $[\delta^1\phi_{\lambda,\Delta}]\cup \alpha_2$ as well as the Massey triple product $\langle \alpha_1, [\delta^1\phi_{\lambda,\Delta}], \alpha_2 \rangle$ are trivial.
\end{theorem}

\begin{proof}
	We consider the Cayley graph $\Cayl(F,\mathcal{P})$ together with the left translation action of $F$. This is a well-defined Cayley graph, since $\mathcal{P}$ is a generating set of $F$ by Property~\ref{prop:Delta_genset} of $\Delta$-decompositions (Definition~\ref{def:Deltadecomp}).
	
	For two vertices $x,y$ in $\Cayl(F;\mathcal{P})$ let $\Delta(x^{-1}y) = (s_1,\ldots, s_n)$. Then we define the path
	\begin{align*}
		p_{xy}\coloneqq x,\ xs_1,\ xs_1s_2,\ \ldots,\ xs_1\cdots s_n
	\end{align*}
	from $x$ to $y$. Let $P(x,y) = \lbrace p_{xy}\rbrace $ and consider the family $P= \big( P(x,y)\big)_{x,y\in F}$. The properties of $\Delta$-decompositions can be directly transferred to show that this family fulfils the quasi-median property for the constant $R$ that bounds the length of the $r$-parts of $\Delta$-triangles. One can check that $P$ is coherent using the properties of $\Delta$-decompositions.
	
	Moreover, we consider the coherent pair $(P,\Phi)$ of size~$1$, where $\Phi$ contains the identity map on $p_{xy}^{(1)}$ for all $x,y\in V$ . In this case, $1$-fragments of a path~$p$ are simply edges of this path. Moreover, we again do not need to worry about path-independence or $\Phi$-stability.
	
	We consider the map 
	\begin{align*}
		\mathcal{W}\colon \Eor &\to \R\\
		e&\mapsto \lambda\big(\alpha(e)^{-1}\omega(e)\big)
	\end{align*}
	and show that it is a $1$-weight. For this, we check that $\mathcal{W}$
	\begin{enumerate}
		\item is $\Gamma$-invariant;
		\item is alternating as $\lambda$ is alternating;
		\item is bounded, since $\lambda$ is a bounded map;
		\item is clearly path-independent, as $\Phi$ only contains identity maps;
		\item fulfils the finiteness property, as $1$-fragments contain no vertex.
	\end{enumerate}
	
	Hence, the map 
	\begin{align*}
		f_{\lambda,\Delta}\colon F^2 &\to \R\\
		(x,y) &\mapsto \sum_{\mathclap{a\in p_{xy}^{(1)}}} \mathcal{W}(a)
	\end{align*}
	is a quasimorphism of $F\curvearrowright \Cayl(F;\mathcal{P})$. Using Remark~\ref{rem:conn_quasimorphisms}, we see that $\phi_{\lambda,\Delta} =  (f_{\lambda,\Delta})_e$ is a quasimorphism of $F$. 
	
	The triviality of the products can be proved in the same way as for Brooks quasimorphisms.
\end{proof}

\subsection{Median quasimorphisms}
\label{subsec:Median}

Finally, we use weight quasimorphisms and Theorems~\ref{main:trivialcupproduct} and~\ref{main:trivialMassey} about the vanishing of their products in order to {(re-)prove} results for median quasimorphisms of group actions on $\cat0$ cube complexes. Median quasimorphisms in this paper are Brooks-type quasimorphisms counting occurrences of a certain segment of halfspaces.

As we generalize the results of Br\"uck, Fournier-Facio and L\"oh~\cite{BFFL}, we will only recall the definitions that are strictly necessary to understand median quasimorphisms and vanishing results.
For readers who are not familiar with $\cat0$ cube complexes and halfspaces, we refer to Br\"uck, Fournier-Facio and L\"oh~\cite{BFFL} for details and many pictures.

\subsubsection*{\texorpdfstring{${\cat0}$}{CAT(0)} cube complexes}
Let $X$ be a $\cat0$ cube complex, i.e., a simply connected cubical complex such that links of vertices are flag complexes.
We denote by $V$ the set of vertices of $X$. Note that $V$ is just the $0$-skeleton $X(0)$ of $X$. Furthermore, the $1$-skeleton of $X$ forms an undirected simplicial graph with vertex set $V$. We denote by $E$ the set of edges of the graph given by the 1-skeleton of $X$. We equip this graph with the graph metric $D$ where edges have length $1$. Note that $(V,E)$ is a median graph~\cite[Theorem~1.18]{Hagen}.

The dimension of the $\cat0$ cube complex $X$ is defined as the highest dimension of a cube in $X$. We will always work with finite dimensional cube complexes.
In the following, let $\Gamma$ be a group.

\begin{definition}[halfspaces]
	Let $X$  be a $\cat0$ cube complex with corresponding median graph $(V,E)$ and graph metric $D$. 
	\begin{itemize}
		\item For an edge $e \in E$ that consists of two distinct vertices $\alpha$ and $\omega$ we define a map
		\begin{align*}
			g_e\colon V&\to \lbrace \alpha, \omega\rbrace \\
			x & \mapsto  \begin{cases} \alpha, & \textrm{if }  D(x,\alpha)<D(x,\omega),  \\
				\omega, & \textrm{if } D(x,\omega)<D(x,\alpha),
			\end{cases}
		\end{align*} 
		called the \emph{gate map} of the closest-point projection for $e$.  (This map is well-defined since $(V,E)$ is a median graph.)
		Each of the sets $g_e^{-1}(\alpha)$ and $g_e^{-1}(\omega)$ is called a \emph{halfspace}. 
		We say that $e$ is \emph{dual} to the halfspaces $g_e^{-1}(\alpha)$ and~$g_e^{-1}(\omega)$. Note that two edges can be dual to the same halfspaces.
		We define $X_\Half$ to be the set of halfspaces of $X$. 
		\item 	Two halfspaces $h_1$ and $h_2$ are said to be \emph{transverse} if each of the four intersections $h_1\cap h_2,\ h_1\cap h_2,\ h_1\cap h_2$, and $h_1\cap h_2$ is non-empty. We then write $h_1\pitchfork h_2$.
		\item  For a halfspace $h\in X_\Half$ we denote by $\overline{h}$ its complement $\overline{h}\coloneqq V\setminus h$. It is clear by definition that $\overline{h}$ is also a halfspace.
		\item For two vertices $x,y\in V$ we say a halfspace $h$ \emph{separates $y$ from $x$} if $y\in h$ and $x\in \overline{h}$ and define the \emph{$\Half$-interval}  $[x,y]_\Half$ to be the set consisting of all halfspaces separating $y$ from $x$.
	\end{itemize}
\end{definition}

\begin{definition}[tightly nested, segment, interior]
	Let $X$ be a $\cat0$ cube complex.
	\begin{itemize}
		\item Two halfspaces $h_1\supset h_2$ of $X$ are called \emph{tightly nested} if there is no other halfspace $h$ such that $h_1\supset~h\supset~h_2$ and $h_1\neq h\neq h_2$.
		An \emph{$\Half$-segment} of length $\ell\in\N$ is a sequence $(h_1\supset\cdots\supset h_\ell)$ of tightly nested halfspaces. 
		\item We denote by $X_\Half^{(\ell)}$ the set of all $\Half$-segments of length $\ell$. The \emph{reverse} $\overline{s}$ of an $\Half$-segment $s= (h_1\supset \ldots\supset h_\ell)$ is defined by $\overline{s} \coloneqq (\overline{h_\ell}\supset \ldots\supset \overline{h_1})$ and it is again an $\Half$-segment.
		\item The action of $\Gamma$ on $X$ via combinatorial automorphisms induces an action of~$\Gamma$ on $X_\Half^{(\ell)}$. The orbit of $s\in X_\Half^{(\ell)}$ is denoted by $\Gamma s$ and we call its elements \emph{translates of $s$.} 	
		\item For vertices $x,y$ of $X$ we write $[x,y]_\Half^{(\ell)}$ for the set of all segments of length $\ell$ whose halfspaces are contained in $[x,y]_\Half$.
		\item We say that a vertex $x$ of $X$  \emph{lies in the interior} of an $\Half$-segment $(h_1\supset \ldots\supset h_\ell)$ if $x\in h_1\cap \overline{h_\ell}$.
	\end{itemize}
\end{definition}

We are now able to describe median quasimorphisms.

\begin{definition}[median quasimorphism]
	\label{def:medianqm}
	Let $\Gamma\curvearrowright X$ be an action via combinatorial automorphisms on a $\cat0$ cube complex $X$. 
	Let $s\in X_\Half^{(\ell)}$ be an $\Half$-segment of length $\ell\in\N$. We define the \emph{median quasimorphism $f_s\colon V\times V\to \R$ for $s$} for all $(x,y)\in V\times V$ to be the number of translates of $s$ in $[x,y]_\Half$ minus the number of translates of $s$ in $[y,x]_\Half$.
\end{definition}

\begin{remark}
	Under the assumptions of Definition~\ref{def:medianqm}, let $s\in X_\Half^{(\ell)}$. If $\Gamma s = \Gamma \overline{s}$, then the corresponding median quasimorphism~$f_s$ is the zero-map.
	If $\Gamma s \neq \Gamma \overline{s}$, we can give a concrete formula for $f_s$ via the following:
	Consider the map
	\begin{align*}
		\epsilon_s \colon X_\Half^{(\ell)} &\to \lbrace -1,0,+1\rbrace \\
		t&\mapsto \begin{cases}
			+1,& \text{if } \Gamma t = \Gamma s,\\
			-1 & \text{if } \Gamma t = \Gamma \overline{s},\\
			0 & \text{otherwise.}
		\end{cases}
	\end{align*}
	Then $f_s$ is given by 
	\begin{align*}
		f_s\colon V\times V&\to \R\\
		(x,y)& \mapsto \sum_{t\in [x,y]_\Half^{(\ell)}}\epsilon_s(t).
	\end{align*}
\end{remark}

\begin{remark}
	Median quasimorphisms are a generalisation of Brooks quasimorphisms. For a non-abelian free group $F$ with free generating set $S$, the Cayley graph $\Cayl(F,S)$ is a tree and hence a $\cat0$ cube complex on which $F$ acts via left translation. For $\omega\in F$ a non-trivial element we consider the $\Half$-segment $s$ of $X$ given by the consecutive edges of the unique geodesic from $e$ to $\omega$. Then the pullback $(f_s)_e$ of the median quasimorphism coincides with the Brooks quasimorphism $\phi_\omega$. Fern\'{o}s, Forester and Tao~\cite{FFT} also introduced a version of median quasimorphisms generalising small Brooks quasimorphisms. These versions only count non-overlapping occurrences of a $\Half$-segment $s$. 
\end{remark}

Br\"uck, Fournier-Facio and L\"oh observed that median quasimorphisms are not always quasimorphisms of the action on the $\cat0$ cube complex, see~\cite[Example~3.12]{BFFL}. This is why they used the following finiteness assumption on $\cat0$ cube complexes which was introduced by Fioravanti~\cite[Definition~4.14]{Fior}:

\begin{definition}
	Let $X$ be a finite dimensional $\cat0$ cube complex and $\sigma\in\N$. A \emph{length-$\sigma$ staircase} in the $\cat0$ cube complex $X$ is a tuple $(h_1\supset\cdots\supset h_\sigma, k_1\supset\cdots\supset k_\sigma)$ of proper chains of halfspaces with the following properties: 
	\begin{itemize}
		\item For all $i\in\lbrace 1,\ldots,\sigma\rbrace$ and all $j\in\lbrace 1,\ldots,i-1\rbrace$, we have $h_i\pitchfork k_j$
		\item For all $i\in\lbrace 1,\ldots,\sigma\rbrace$, we have $h_i\supsetneq k_i$.
	\end{itemize}
	The \emph{staircase length of $X$} is the maximal length of a staircase in $X$ (or else we set it to be $\infty$).
\end{definition}

\subsubsection*{Median quasimorphisms and vanishing results}
The following can be proved with Theorem~\ref{main:quasimorphi} about weight quasimorphisms:

\begin{theorem}[{\cite[Proposition~3.16]{BFFL}}]
	\label{thm:medianqm}
	Let $X$ be a finite dimensional $\cat0$ cube complex of finite staircase length $\sigma$ and let $\Gamma$ act on $X$ via combinatorial automorphisms. Let $\ell\in \N$ and let $s\in X^{(\ell)}_\Half$ be a $\Half$-segment. Then $\delta^1f_s\colon V^3 \to\R$ is bounded. 
	This means that $f_s$ is a quasimorphism of $\Gamma\curvearrowright X$.
\end{theorem}

\begin{proof}
	With a slight abuse of notation, we also denote by $X=(V,E)$ the underlying median graph given by the 1-skeleton of $X$.
	Every oriented edge $e=(\alpha(e),\omega(e))$ is dual to exactly one halfspace $h_e$ that contains $\omega(e)$.This gives rise to a map 
	\begin{align*}
		\lambda\colon \Eorl&\to X_\Half^\ell\\
		(e_1,\ldots,e_\ell)&\mapsto (h_{e_1},\ldots,h_{e_\ell}).
	\end{align*}
	Note that the image of this map might contain sequences that are not tightly nested. For $x,y\in V$ and $p\in [x,y]$, we consider the restriction 
	\begin{align*}
		\lambda|_{ p^{(\ell)}} \colon p^{(\ell)} \to [x,y]_\Half^{\ell}.
	\end{align*}
	At first, we check that the image of this map indeed lies in $[x,y]_\Half^{\ell}$: For an edge $e=(\alpha,\omega)$ of $p$ with the orientation given by $p$ we have $x\notin h_e$ and $y\in h_e$. Hence, we deduce $h_e\in[x,y]_\Half$. Furthermore, this map is injective, as there is a 1-to-1 correspondence between the edges on $p$ and $[x,y]_\Half$, see \cite[Lemma~3.2]{BFFL}.
	
	We now define
	\begin{align*}
		\mathcal{W}\colon \Eorl&\to \lbrace+1,-1,0\rbrace\\
		a&\mapsto \begin{cases}
			\epsilon_s(\lambda(a))& \textrm{if } \lambda(a)\in X_\Half^{(\ell)},\\
			0& \textrm{otherwise}. 
		\end{cases}
	\end{align*}
	
	We want to prove that $\mathcal{W}$ is an \weight\ so that we can conclude by showing that the median quasimorphism $f_s$ corresponds with the weight quasimorphism $f_\mathcal{W}$.
	At first we construct a family $P=(P(x,y))_{x,y\in V}$ of paths by $P(x,y) =[x,y]$.
	As the 1-skeleton of a $\cat0$ cube complex forms a median graph we can deduce that $P$ fulfils the quasi-median property for $R=1$. Furthermore, it is coherent by the properties of geodesics. Next, we fix two vertices $x,y\in V$. 
	
	For two geodesics $p,q$ from $x$ to $y$, we consider the injective map
	\begin{align*}
		\lambda|_{ p^{(\ell)}} \colon p^{(\ell)} \to [x,y]_\Half^{\ell}.
	\end{align*}
	We prove $[x,y]_\Half^{(\ell)} \subset \im (\lambda|_{ p^{(\ell)}})$: For~$(h_1\supset\cdots\supset h_\ell)\in [x,y]_\Half^{(\ell)}$ and the corresponding oriented edges $e_1,\ldots,e_\ell$ on $p$, it is clear that for $i,j\in \lbrace 1,\ldots, \ell\rbrace$ with $i<j$, the edges $e_i$ and $e_j$ appear in this order on $p$, as the dual halfspaces fulfil the relation $h_i\supset h_j$. So $(e_1,\ldots,e_\ell)\in p^{(\ell)}$.
	
	This allows us to fix a bijection
	\begin{align*}
		\varphi_{p,q}\colon p^{(\ell)} \to q^{(\ell)} 
	\end{align*} 
	that satisfies $\varphi_{p,q}(t) =t'$ if $\lambda(t) = \lambda(t')\in [x,y]_\Half^{(\ell)}$.
	In particular, this means for $t\in p^{(\ell)}$ that $\lambda(t)\in [x,y]_\Half^{(\ell)}$ if and only if~$\lambda(\varphi_{p,q}(t))\in [x,y]_\Half^{(\ell)}$.
	
	We fix
	\begin{align*}
		\Phi = (\varphi_{p,q})_{(p,q)\in \bigcup\limits_{{(x,y)\in V^2}} P(x,y)^2}
	\end{align*}
	and obtain a coherent pair $(P,\Phi)$ of size $\ell$.
	
	It remains to prove that $\mathcal{W}$ is an \weight. For this we check that $\mathcal{W}$ 
	\begin{enumerate}
		\item is $\Gamma$-invariant, since the map $\epsilon_s$ is $\Gamma$-invariant;
		\item is alternating, since $\epsilon_s$ is alternating on sequences of halfspaces;
		\item is bounded by $1$;
		\item is path-independent with respect to $\Phi$: Let $x,y\in V$, $p,q\in P(x,y)$ and $t\in p^{(\ell)}$. If $\lambda(t)\notin [x,y]_\Half^{(\ell)}$, it is $\lambda(\varphi_{p,q}(t))\notin [x,y]_\Half^{(\ell)}$ and hence, $\mathcal{W}(t) = \mathcal{W}(\varphi_{p,q}(t)) = 0$. On the other hand, if $\lambda(t)\in [x,y]_\Half^{(\ell)}$, we know that $\lambda(\varphi_{p,q}(t))=\lambda(t)$ and hence, $\mathcal{W}(t) = \mathcal{W}(\varphi_{p,q}(t))$; and 
		\item fulfils the finiteness condition: 
		Let $x,y\in V$, $p\in P(x,y)$ and $m$ be a vertex on $p$. At first, we notice that an element $a\in p^{(\ell)}\cap \supp(\mathcal{W})$ fulfils $\lambda(a)\in [x,y]_\Half^{(\ell)}$ and it contains $m$ in its interior if $\lambda(a)$ contains $m$ in its interior. Using that the $\cat0$ cube complex has finite dimension and is of finite staircase length allows us to conclude that there exists a constant $c\in \N$ independent of the choice of $x,y$ and the path $p\in P(x,y)$ such that there are at most $c$ such segments in~$[x,y]_\Half^{(\ell)}$~\cite[Lemma~3.15]{BFFL}. The injectivity of $\lambda$ on $p^{(\ell)}$ proves the finiteness condition.
	\end{enumerate}
	
	Let $f_\mathcal{W}$ be the corresponding weight quasimorphism. It remains to prove $f_s=f_\mathcal{W}$. For this, let $x,y\in V$ and $p_{xy}\in[x,y]$. Then we compute 
	\begin{align*}
		f_{\mathcal{W}}(x,y) = \sum_{\mathclap{a\in p_{xy}^{(\ell)}}}\mathcal{W}(a) = \sum_{\mathclap{a\in p_{xy}^{(\ell)}}}\epsilon_s(\lambda(a) ) =  \sum_{\mathclap{{t\in [x,y]_\Half^{(\ell)}}}}\epsilon_s(t) = f_s(x,y).
	\end{align*}
	The third equality uses 
	\begin{itemize}
		\item the injectivity of $\lambda$,
		\item the fact that $\lambda(p_{xy}^{(\ell)}) \subset [x,y]^\ell$, and
		\item that the support of $\epsilon_s$ is contained in $X_\Half^{(\ell)}$.
	\end{itemize}
\end{proof}

\begin{remark}
	\label{rem:why}
	As we see in Figure~\ref{subfig:segmentvsfragment}, for two vertices $u,v\in V$ and a geodesic $p\in[x,y]$, a $\Half$-segment in $ [x,y]^{(l)}$ might correspond to a fragment of $p^{(l)}$ whose elements are not directly consecutive edges.
	This is why we constructed weight quasimorphisms as weighted sums over \emph{fragments} of paths. 
	
	We also want to mention that for $x,y\in V$ and $p,q\in[x,y]$ there is an obvious bijection between the edges of $p$ and $q$ that assigns two edges to each other if their corresponding halfspaces coincide. Unfortunately, this map on the level of edges does not induce a well defined map $p^{(l)}\to q^{(l)}$ since it is not necessarily order preserving, see Figure~\ref{subfig:segment}. This is why we needed to fix bijections on the level of \step s in Definition~\ref{def:coherentpair} of coherent pairs. Furthermore, this example also explains why the definition of the map $\varphi_{pq}$ in the proof of Theorem~\ref{thm:medianqm} needs to be so abstract. 
\end{remark}

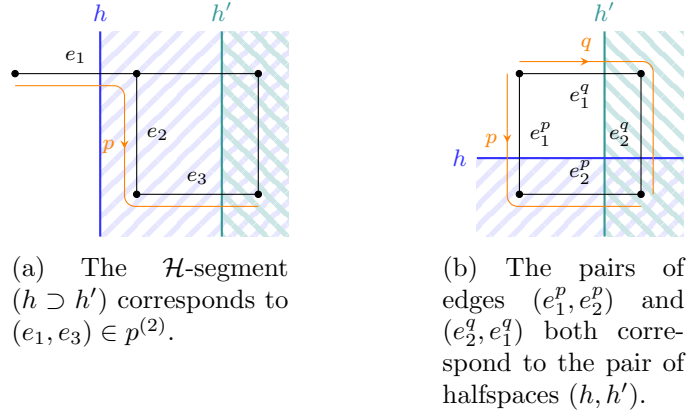
\begin{figure}[h] 
	\begin{center}
		\subcaptionbox{The $\Half$-segment $(h\supset h')$ corresponds to $(e_1,e_3)\in p^{(2)}$.%
			\label{subfig:segmentvsfragment}}{%
			\begin{tikzpicture}[>=stealth, scale = 0.8]
				\newcommand{\arrowIn}{\tikz \draw[-stealth] (-1pt,0) -- (1pt,0);}
				
				\fill [pattern={Lines[
					distance=2mm,
					angle=45,
					line width=0.7mm
					]},
				pattern color=blue!10
				] (-1.6,1.7) rectangle (1.5,-1.7);
				\draw[-,blue!80,thick] (-1.6,1.7) to (-1.6,-1.7);
				\node[above, blue!80] at (-1.6,1.7) {\footnotesize$h$};
				
				\fill [pattern={Lines[
					distance=2mm,
					angle=-45,
					line width=0.7mm
					]},
				pattern color=teal!20
				] (0.4,1.7)  rectangle (1.5,-1.7);
				\draw[-,teal!80,thick] (0.4,1.7) to (0.4,-1.7);
				\node[above, teal!80] at (0.4,1.7) {\footnotesize$h'$};	
				
				\coordinate (a) at (-1,-1);
				\coordinate (b) at (1,-1);
				\coordinate (c) at (1,1);
				\coordinate (d) at (-1,1);
				\coordinate (e) at (-3,1);
				
				\filldraw (a) circle (1.5pt);
				\filldraw (b) circle (1.5pt);
				\filldraw (c) circle (1.5pt);
				\filldraw (d) circle (1.5pt);
				\filldraw (e) circle (1.5pt);
				
				\draw[-] (-1,-1) to node[right] {\scriptsize$e_2$} (-1,1);
				\draw[-] (-1,-1) to node[above] {\scriptsize$e_3$} (1,-1);
				\draw[-] (1,-1) to (1,1);
				\draw[-] (-1,1) to (1,1);
				\draw[-] (-1,1) to node[above] {\scriptsize$e_1 $} (-3,1);
				\draw [orange, rounded corners] (-3,0.8) --  (-1.2,0.8) -- node[sloped, left, rotate=90] {\scriptsize$p$} node[sloped, pos=0.5, allow upside down]{\arrowIn} (-1.2, -1.2) -- (1,-1.2); 
		\end{tikzpicture}}
		\hfil
		\subcaptionbox{The pairs of edges $(e_1^p,e_2^p)$ and $(e_2^q,e_1^q)$ both correspond to the pair of halfspaces $(h,h')$.\label{subfig:segment}}{%
			\begin{tikzpicture}[>=stealth, scale = 0.8]
				\newcommand{\arrowIn}{\tikz \draw[-stealth] (-1pt,0) -- (1pt,0);}
				
				\fill [pattern={Lines[
					distance=2mm,
					angle=45,
					line width=0.7mm
					]},
				pattern color=blue!10
				] (1.7,-0.4) rectangle (-1.7,-1.7);

				\fill [pattern={Lines[
					distance=2mm,
					angle=-45,
					line width=0.7mm
					]},
				pattern color=teal!20
				] (0.4,1.7)  rectangle (1.7,-1.7);
				
				\draw[-,teal!80,thick] (0.4,1.7) to (0.4,-1.7);
				\node[above, teal!80] at (0.4,1.7) {\footnotesize$h'$};	
				
				\draw[-,blue!80,thick] (1.7,-0.4) to (-1.7,-0.4);
				\node[left, blue!80] at (-1.7,-0.4) {\footnotesize$h$};	
				
				\coordinate (a) at (-1,-1);
				\coordinate (b) at (1,-1);
				\coordinate (c) at (1,1);
				\coordinate (d) at (-1,1);
				
				\filldraw (a) circle (1.5pt);
				\filldraw (b) circle (1.5pt);
				\filldraw (c) circle (1.5pt);
				\filldraw (d) circle (1.5pt);
				
				\draw[-] (-1,-1) to node[right] {\scriptsize$e_1^p$} (-1,1);
				\draw[-] (-1,-1) to node[above] {\scriptsize$e_2^p$} (1,-1);
				\draw[-] (1,-1) to node[left] {\scriptsize$e^q_2$} (1,1);
				\draw[-] (-1,1) to node[below] {\scriptsize$e^q_1$} (1,1);
				\draw [orange, rounded corners] (-1.2,1) -- node[sloped, left, rotate=90] {\scriptsize$p$} node[sloped, pos=0.5, allow upside down]{\arrowIn} (-1.2, -1.2) -- (1,-1.2); 
				\draw [orange, rounded corners]  (-1,1.2) -- node[sloped, above] {\scriptsize$q$} node[sloped, pos=0.5, allow upside down]{\arrowIn} (1.2, 1.2) -- (1.2,-1); 
		\end{tikzpicture}}
	\end{center}
	\caption{Connection between $\Half$-segments and fragments of paths.} 
\end{figure}

Br\"uck, Fournier-Facio and L\"oh proved a vanishing result for cup products with classes given by median quasimorphism. However, they observed that cup products with classes induced by median quasimorphisms do not vanish in general~\cite[Example~3.19]{BFFL}. This then led to the notion of non-transverse classes.

\begin{definition}[heads/tails]
	Let $s = (h_1 \supset \cdots\supset h_\ell)\in X^{(\ell)}$. We say that $\alpha\in V$ is a \emph{head} of $s$ if $\alpha\in \overline{h_1}$ and there exists an edge dual to $h_1$ that has $\alpha$ as one of its endpoints. We say that $\omega\in V$ is a \emph{tail} of $s$ if $\omega\in h_\ell$ and there exists an edge dual to $h_\ell$ that has $\omega$ as one of its endpoints.
	We let $\alpha(s)$ denote the set of heads of $s$ and we let $\omega(s)$ denote the set of tails of~$s$. By definition, $\alpha(\overline{s}) = \omega(s)$ and $\omega(\overline{s}) = \alpha(s)$.
\end{definition}

\begin{definition}[non-transverse]
	Let $s\in X_\Half^{(\ell)}$ and let $\kappa\in C^n_{\Gamma,b}(X;\R)$. We say that $\kappa$ and $s$ are \emph{non-transverse} if for all $x_1,\ldots,x_n\in V$, the value of $\kappa(\alpha,x_1,\ldots,x_n)$ is constant over all $\alpha\in\alpha(s)$, and the value of $\kappa(\omega,x_1,\ldots,x_n)$ is constant over all $\omega \in \omega(s)$.
	
	Given a set $S\subset X^{(\ell)}_\Half$ we say that $\kappa$ and $S$ are \emph{non-transverse} if $\kappa$ and $s$ are non-transverse for all $s\in S$.
	
	We say that a class in equivariant bounded cohomology has this property if it admits a representative that does.
\end{definition}

We are now able to prove the following theorem, extending the results of Br\"uck, Fournier-Facio, and L\"oh~\cite[Theorem~3.23]{BFFL}.

\begin{theorem}
	\label{thm:median}
	Let $X$ be a finite dimensional $\cat0$ cube complex of finite staircase length $\sigma$. Let $s\in X_\Half^{(l)}$ be tightly nested. 
	
	If $\alpha_1\in H_{\Gamma,b}^{n}(X;\R)$ and $\alpha_2\in H_{\Gamma,b}^{m}(X;\R)$ are non-transverse to $\Gamma s$, then the cup products  $\alpha_1 \cup [\delta^1f_s]$ and $[\delta^1f_s]\cup \alpha_2$ as well as the Massey triple product $\langle \alpha_1, [\delta^1f_s], \alpha_2 \rangle$ are trivial.  
\end{theorem}

\begin{proof}
	Again by a slight abuse of notation we denote by $X=(V,E)$ the median graph given by the 1-skeleton of $X$.
	We consider the coherent family $(P,\Phi)$ and the \weight\ $\mathcal{W}$ as in the proof of Theorem~\ref{thm:medianqm}. Moreover, we recall that in this setting, the median quasimorphism $f_s$ corresponds with the weight quasimorphism for $\mathcal{W}$.
	For the vanishing of the products, we need to show that every class~$\alpha\in H_{\Gamma,b}(X;\R)$ that is non-transverse to $\Gamma s$ is also \stable. Then, the vanishing results follow from Theorems~\ref{main:trivialcupproduct} and~\ref{main:trivialMassey}.
	
	Let $\zeta\in C_{\Gamma,b}^n(X;\R)$ be a representative of $\alpha$ that is non-transverse to $\Gamma s$, i.e., for all $\gamma\in \Gamma$ and $x_1,\ldots,x_n\in V$, the values
	\begin{align*}
		\zeta(\alpha, x_1,\ldots,x_n) \text{ and } \zeta(\omega,x_1,\ldots,x_n)
	\end{align*}
	do not depend on the choice of head $\alpha\in\alpha(\gamma s)$ and tail $\omega\in \omega(\gamma s)$.
	In order to show that $\zeta$ is \stable, let $x,y\in V$ and $p,q\in P(x,y)$. Furthermore, let $x_1,\ldots,x_n\in V$. If $a\in p^{(\ell)}\cap\supp(\mathcal{W})$ we know that $\lambda(a) = \lambda(\varphi_{p,q}(a))$ and so $\alpha(a)$ and $\alpha(\varphi_{p,q}(a))$ are both heads of the $\Half$-segment $\lambda(a)$.
	This shows
	\begin{align*}
		\zeta(\alpha(a),x_1,\ldots,x_n) = \zeta(\alpha(\varphi_{p,q}(a)),x_1,\ldots,x_n).
	\end{align*}
	The same arguments show the equality for the tails. So $\zeta$ is a \stable\ cocycle.
\end{proof}

\begin{remark}
	Apart from the finite staircase length, there is a second option to obtain finiteness results for computations in $\cat0$ cube complexes, which is using \"uber-parallel nested tuples of halfspaces instead of tightly nested ones~\cite{CFI}. Considering median quasimorphisms that count occurences of a fixed \"uber-parallel tuple, one should also be able to obtain similar results by using the finiteness results developed by Chatterji, Fern\'os, and Iozzi~\cite[Proof of Proposition~3.4]{CFI}.
\end{remark}



\begin{thebibliography}{99}
		\addcontentsline{toc}{chapter}{Bibliography}
		
		\bibitem{AB}
		S.~Amontova and M.~Bucher.
		Trivial cup products in bounded cohomology of the free group via aligned chains, 
		\emph{Forum Math.}, 34(4), 933-943, 2022.
		
		\bibitem{Bowditch}
		B.~Bowditch.
		Coarse median spaces and groups.
		\emph{Pacific F. Math.}, 261(1):53-93, 2013.
				
		\bibitem{Brooks}
		R.~Brooks.
		Some remarks on bounded cohomology.
		In \emph{Riemann Surfaces Related Topics (AM-97), Volume 97},
		pages 53-64. Princeton University Press, 2016.
		
		\bibitem{BFFL}
		B.~Br\"uck, F.~Fournier-Facio and C.~L\"oh.
		Median quasimorphisms on CAT(0) cube complexes and their cup products,
		In \emph{Geom. Dedicata} 218(1), Paper No. 28, 2024.
		
		\bibitem{BM}
		M.Bucher and N.~Monod.
		The cup product of Brooks quasimorphisms.
		\emph{Forum Mathematicum} 30.5, 2018.
		
		\bibitem{BIMW}
		 M.~Burger, A.~Iozzi, N.~Monod, and A.~Wienhard. 
		 Bounds for cohomology classes. 
		 \emph{Enseign. Math. (2)}, 54(1-2):3-189, 2008. 
		 Guido's book of conjectures, A gift to Guido Mislin on the occasion of his retirement from ETHZ, June 2006, collected by I. Chatterji.
		
		\bibitem{CFI}
		I.~Chatterji, T.~Fern\'{o}s, and A.~Iozzi. 
		The median class and superrigidity of actions on CAT(0) cube complexes.
		\emph{J. Topol.}, 9(2):349-400, 2016. 
		With an appendix by Pierre-Emmanuel Caprace.
		
		\bibitem{EF}
		D.~B.~Epstein and K.~Fujiwara.
		The second bounded cohomology of word-hyperbolic groups.
		\emph{Topology}, 36(6):1275-1289, 1997.
		
		\bibitem{FFT}
		T. Fern\'{o}s, M. Forester, and J. Tao. 
		Effective quasimorphisms on right-angled Artin groups. 
		\emph{Ann. Inst. Fourier}, 69(4):1575-1626, 2019.
		
		\bibitem{Fior}
		E.~Fioravanti.
		Coarse-median preserving automorphisms,
		\emph{Geom. Topol.}, 28:161-266 , 2024.

		\bibitem{FF}
		F.~Fournier-Facio.
		Infinite sums of Brooks quasimorphisms and cup products in bounded cohomology.
		\textsf{arXiv:2002.10323} [math.GR], 2020.	

		\bibitem{FG}
		F.~Fournier-Facio, A.~Genevois.
		No quasi-isometric rigidity for proper actions on CAT(0) cube complexes.
		\emph{Proc. Amer. Math. Soc.}, 151(12):5097-5109, 2023.

		\bibitem{frigerio}
		R.~Frigerio.
		\emph{Bounded cohomology of discrete groups},
		volume 227 of \emph{Mathematical Surveys and Monographs}.
		American Mathematical Society, Providence, RI, 2017.
		
		\bibitem{Genevois}
		A.~Genevois.
		Algebraic properties of groups acting on median graphs.
		Available at \url{https://sites.google.com/view/agenevois/books}, accessed June 2024. 
		
		\bibitem{Grigorchuk}
		R.~I.~Grigorchuk.
		 Some results on bounded cohomology. 
		 In \emph{Combinatorial and geometric group theory (Edinburgh, 1993)}, 
		 volume 204 of \emph{London Math. Soc. Lecture Note Ser.}, pages 111-163. 
		 Cambridge Univ. Press, Cambridge, 1995.
		
		\bibitem{Hagen}
		M.~Hagen.
		CAT(0) cube complexes, median graphs, and cubulating groups.
		Available at \url{https://www.wescac.net/into_the_forest.pdf},
		accessed June 2024.
		
		\bibitem{Heuer}
		N.~Heuer.
		Cup product in bounded cohomology of the free group.
		\emph{Ann. Sc. Norm. Super. Pisa Cl. Sci. (5)}, 21, 1-26, 2020.
		Accessed version: \textsf{arXiv:1710.03193v2 }[math.GR], 2018.
		
		\bibitem{Li}
		K.~Li.
		Bounded cohomology of classifying spaces for families of subgroups,
		\emph{Algebr. Geom. Topol.},
		23(2):933-962, 2023.
	
		\bibitem{Loeh}
		C.~L\"oh.
		\emph{Geometric Group Theory. An Introduction},
		Universitext, Springer, 2017.

		\bibitem{Marasco}
		D.~Marasco.
		Trivial Massey Product in Bounded Cohomology,
		\textsf{arXiv:2209.00560} [math.GR], 2022.
		
		\bibitem{Monod}
		N.~Monod.
		\emph{Continuous bounded cohomology of locally compact groups},
		volume 1758 of \emph{Lecture notes in Mathematics}.
		Springer, 2001.
		
		\bibitem{Rud}
		Y.~B.~Rudyak.
		On category weight and its applications.
		In \emph{Topology} 38(1):37-55, 1999.
		
		\bibitem{Soma}
		T.~Soma.
		Bounded cohomology and topologically tame Kleinian groups.
		\emph{Duke Math. J.},
		88(2):357-370,1997.
		
		\bibitem{Viktor}
		V.~V.~Prasolov.
		\emph{Elements of Homology Theory},
		volume 81 of \emph{Graduate Studies in Mathematics}.
		American Mathematical Society, Providence, RI, 2007.
	\end{thebibliography}
\end{document}